\documentclass[10pt]{elsarticle}
\usepackage{lineno,hyperref}
\usepackage{booktabs}
\usepackage{multirow}
\usepackage{amsthm}
\usepackage{graphicx}
\usepackage{subfigure}
\usepackage{float}
\usepackage{bm}
\usepackage{lipsum}
\usepackage[marginal]{
	footmisc}
\usepackage{listings} 
\usepackage{color}
\usepackage{amsfonts,enumerate,amsmath,amssymb}
\def\qed{\strut\hfill $\Box$}
\newtheorem{thm}{Theorem}[section]

\newtheorem{prop}[thm]{Proposition}
\newtheorem{lem}[thm]{Lemma}

\newtheorem{rem}[thm]{Remark}
\newtheorem{defn}[thm]{Definition}

\newcommand{\thmref}[1]{Theorem~{\rm \ref{#1}}}
\newcommand{\lemref}[1]{Lemma~{\rm \ref{#1}}}
\newcommand{\propref}[1]{Proposition~{\rm \ref{#1}}}

\def\para#1{\vskip .4\baselineskip\noindent{\bf #1}}
\numberwithin{equation}{section}
\begin{document}
	\begin{frontmatter}

		\title{Large deviation principle for slow-fast system with mixed fractional Brownian motion}
		
		\author[mymainaddress]{Yuzuru INAHAMA\footnote[1]{The author's names are in alphabetical order, following a traditional custom of the mathematical community. The author Yuzuru INAHAMA and Xiaoyu YANG contributed to the work equally.}}
		\ead{inahama@math.kyushu-u.ac.jp}

		\author[mysecondaddress,mythirdaddress]{Yong XU\corref{mycorrespondingauthor}}
		\cortext[mycorrespondingauthor]{Corresponding author}
		\ead{hsux3@nwpu.edu.cn}
		
		\author[mysecondaddress]{Xiaoyu YANG}
		\ead{yangxiaoyu@yahoo.com}

		\address[mymainaddress]{Faculty of Mathematics,  Kyushu University, Fukuoka, 8190395, Japan}
		\address[mysecondaddress]{School of Mathematics and Statistics, Northwestern Polytechnical University, Xi'an, 710072, China}
		
		\address[mythirdaddress]{MOE Key Laboratory of Complexity Science in Aerospace, Northwestern Polytechnical University, Xi’an, 710072, China}
		
		\begin{abstract}
			This work focuses on a  slow-fast system perturbed by mixed fractional Brownian motion with  Hurst parameter $H\in(1/2,1)$. The integral with respect to fractional Brownian motion is the generalized Riemann-Stieltjes integral and the integral with respect to Brownian motion is the standard It\^o integral. Our approach is based on the variational framework and the  weak convergence criteria  for mixed fractional Brownian motion. By combining the weak convergence method and Khasminskii's averaging principle, we show a large deviation principle for the slow component.
 %
			\vskip 0.08in			
			
			\noindent{\bf AMS Math Classification.}
		{60F10, 60G15, 60H10}.
			\vskip 0.08in			\end{abstract}		
	\end{frontmatter}

	\section{Introduction}\label{sec-1}
	In this paper, we consider the following 
	 slow-fast system 
	of stochastic differential equations (SDEs) on the time interval $[0,T]$
driven by Brownian motion (Bm) and  fractional Brownian motion (fBm) in $\mathbb{R}^{m}\times \mathbb{R}^{n}$:
	\begin{eqnarray}\label{1}
	\left
	\{
	\begin{array}{ll}
	dx^{(\varepsilon, \delta)}_t = f_{1}(x^{(\varepsilon, \delta)}_t, y^{(\varepsilon, \delta)}_t)dt + \sqrt \varepsilon  \sigma_{1}( x^{(\varepsilon, \delta)}_t)dB^H_{t},\\
	\delta dy^{(\varepsilon, \delta)}_t = f_{2}( x^{(\varepsilon, \delta)}_t, y^{(\varepsilon, \delta)}_t)dt + {\sqrt \delta}\sigma_2( x^{(\varepsilon, \delta)}_t, y^{(\varepsilon, \delta)}_t)dW_{t}
	\end{array}
	\right.
	\end{eqnarray}
with the (non-random) initial condition 
$(x^{(\varepsilon, \delta)}_0, y^{(\varepsilon, \delta)}_0)=(x_0, y_0)\in \mathbb{R}^{m}\times \mathbb{R}^{n}$.
Here,  the time interval is $[0,T]$, $T>0$, and
$f_1:  \mathbb{R}^{m}\times \mathbb{R}^{n} \rightarrow  \mathbb{R}^{m}$, $f_2:  \mathbb{R}^{m}\times \mathbb{R}^{n} \rightarrow  \mathbb{R}^{n}$, 	 $\sigma_{1}: \mathbb{R}^{m} \rightarrow \mathbb{R}^{m\times d_1}$ and 
$\sigma_{2}: \mathbb{R}^{m}\times \mathbb{R}^{n} \rightarrow \mathbb{R}^{n\times d_2}$ are suitable nonlinear functions. 
Precise conditions on these coefficient functions will be specified later.
	 The driving processes 
	 $B^H =(B^H_t)_{t \in [0,T]}$ and $W =(W_t)_{t \in [0,T]}$ are  $d_1$-dimensional fBm 
	 with Hurst parameter $H \in (1/2, 1)$  
	  and  $d_2$-dimensional standard Bm, respectively,
	      which are assumed to be independent.
	  	   The two small parameters $\varepsilon$ and $\delta$, satisfying $0< \delta \ll \varepsilon \ll 1$, are used to describe the separation of time scale between the slow variables $x^{(\varepsilon, \delta)}$ and the fast variables $y^{(\varepsilon, \delta)}$.	Under suitable conditions, the slow-fast system (\ref{1}) admits a unique (pathwise) solution $(x^{(\varepsilon, \delta)},y^{(\varepsilon,\delta)})\in C^{\alpha}\left([0,T], \mathbb{R}^m\right) \times C\left([0,T], \mathbb{R}^n\right)$ with $\alpha\in (1-H,1/2)$, which will be precisely stated in Section 3.
Here, $C^{\alpha}\left([0,T], \mathbb{R}^m\right)$ and
 $C\left([0,T], \mathbb{R}^n\right)$ are the $\alpha$-H\"older continuous 
 path space and the continuous path space, respectively.  	   
	  	   

  	   	  	   It is not too difficult to see that under some conditions, the effective dynamics of the slow system in \eqref{1} can be described by   the  averaged system efficiently as $\delta \to 0$, which comes from  averaging principle \cite{2023Pei}. Precisely speaking, with fixed slow component, under appropriate conditions, the fast component has an invariant probability measure, it shows that as {$\delta \to 0$}, the slow component converges to an  averaged component defined as follows,  
  	   	  	   \begin{equation}\label{4}
  	   	  	   \left\{\begin{array}{l}
  	   	  	   d \bar{x}_t=\bar{f}_1\left(\bar{x}_t\right) d t \\
  	   	  	   \bar{x}_0=x_0 \in \mathbb{R}^m,
  	   	  	   \end{array}\right.
  	   	  	   \end{equation}
  	   	  	  {where $\bar{f}_1(x)=\int_{\mathbb{R}^{n}} f_1(x, y) d \mu_x(y)$ and $\mu_x$ }is a unique invariant probability measure of the fast component {with the slow variable 
			    being ``frozen" at a deterministic $x \in \mathbb{R}^m$}. 
			   This averaging
  	   	  	   principle provides an efficient approach to reduce  computational complexity \cite{2015Xu,2023Pei,2022Wang}. It can be viewed as a variant of the law of large numbers.

  	   	   Compared with averaging principle, a large deviation principle could capture the dynamical behavior more precisely, in other words, characterizing the exponential decay rate of probabilities of rare events \cite{Dembo2009Large}. The main purpose of  this work is to prove a large deviation principle for the slow component $x^{(\varepsilon,\delta)}$ of the above system \eqref{1}.  
The family
	  	   ${x^{(\varepsilon,\delta)}}$ of $C^{\alpha}\left([0,T], \mathbb{R}^m\right)$-valued random elements is said to satisfy a large deviation principle on $C^{\alpha}\left([0,T], \mathbb{R}^m\right)$ with a
	  	   rate function $I: C^{\alpha}\left([0,T], \mathbb{R}^m\right)\rightarrow [0, \infty]$ ($1-H <\alpha <1/2$) if the following two conditions hold:
	  	   \begin{itemize}
	  	   	\item For each closed subset $F$ of $C^{\alpha}\left([0,T], \mathbb{R}^m\right)$,
	  	   	$$\limsup _{\varepsilon \rightarrow 0} \varepsilon \log \mathbb{P}\big(x^{(\varepsilon,\delta)} \in F\big) \leqslant-\inf _{x \in F} I(x).$$
	  	   	\item For each open subset $G$ of $C^{\alpha}\left([0,T], \mathbb{R}^m\right)$,
	  	   	$$\liminf _{\varepsilon \rightarrow 0} \varepsilon \log \mathbb{P}\big(x^{(\varepsilon,\delta)} \in G\big) \geqslant-\inf _{x \in G} I(x).$$
	  	   \end{itemize}
		  {
This will be shown in our main theorem (Theorem \ref{thm}).  
A concrete definition of $I$, which is actually a good rate function, will also be given there. 	   
	  	}   
Now, the large deviation principle has become  an important topic in the probability field \cite{2010Asymptotic}.  Moreover, large deviations have been widely applied in many fields, such as statistics, information theory, engineering, and so on \cite{2009The}. The history of the large deviation principle for the  above system is long, which can be traced back to the result on the large deviation for stochastic systems  by Freidlin and Wentzell \cite{1984Random}. The  slow-fast system  \eqref{1} including multi-time scales can be widely-used in diverse areas, for instance,  climate-weather interactions models (see Kiefer \cite{2000Kiefer})  within  climate being the slow motion and weather the fast one;   the level of the asset price in financial economics \cite{2001BNS},    laser model \cite{2009Chaotic}, compartmental model \cite{Krupa2008Mixed} and so on. 
 Up to now, there have existed some kinds of  methods studying  large deviations for the stochastic slow-fast system.
The connection between exponential functionals and variational representations
 	has been exploited by Dupuis and Ellis \cite{1997A}. After that, the weak convergence method, based on the variational representation for the nonnegative functional of Bm, are constructed by Bou\'e and Dupuis \cite{1998Dupuis}.
	Then, the weak convergence method has been used intensively \cite{DS,2011Variational,2013Large,2020Large}. After that, the PDE theory \cite{BCS},  nonlinear semigroups and viscosity solution theory has been proposed in \cite{KP,FK, FFK}.

However, the aforementioned references all focused on the slow-fast system with diffusion. Up to now, there just  few work concentrating on the large deviation principle for slow-fast system with mixed fBm. In this work, we aim to address this issue. Different from standard Bm,   fBm is self-similar and possesses long-range dependence, which has been employed  to characterize randomness in some complex systems \cite{2008Biagini}. The Hurst parameter $H$ associated with  fBm describes the raggedness of the resultant motion, with a higher value leading to a smoother motion \cite{2008Mishura}. It is worth mentioning that along the variational representation for random functionals on abstract Wiener spaces \cite{2009Zhang}, the weak convergence method for  fBm was proposed by Budhiraja and Song \cite{2020Budhiraja}. After that, Gailus and Gasteratos established the large deviation for the slow-fast system with fBm by the homogenization theory and viable pair \cite{2022Gailus}. However, it is a non-trivial task to verify the corresponding conditions when applying the homogenization theory.

In our work, the problem is solved by the weak convergence method and   Khasminskii type averaging principle  efficiently.   
Firstly, based on the Bou\'e-Dupuis' variational representation formula for a standard Bm  \cite{BP_book},   the variational representation formula for  mixed fBm is given. Based on this, the study is turned into basic qualitative properties for the controlled system. Particularly,   the Khasminskii type averaging principle plays a key role in the weak convergence for the controlled slow-fast system.    Thanks to the particular regime that $\delta=o(\varepsilon)$, in the limit there is no control in the fast component. Then,  the weak convergence of the controlled slow component is obtained by the    time discretization techniques, the exponential ergodicity of the auxiliary fast component without  controlled term and so on.

	The structure of this paper is as follows. In Section 2, we introduce the notation and some preliminaries. In Section 3, we give assumptions and a precise  statement of our main result (Theorem \ref{thm}). 
	 In Section 4, some preliminary lemmas are proved. The proof of our main result is given in Section 5. 

	Throughout this paper, we set 
	{$\mathbb{N}=\{1,2, \ldots\}$}
	and denote by
	$c$, $C$, $c_1$, $C_1$, etc. certain positive constants that may change from line to line. {The time horizon $T >0$ and the non-random
	 initial value 
	$(x_0,y_0)$ are arbitrary but fixed throughout this paper. We do not keep track of 
	$T$ and $(x_0,y_0)$.}
	
	\section{Preliminaries and Assumptions}\label{sec-2}
	\subsection{Fractional integrals and derivatives}\label{sec-2-1}
	Firstly, recall some basic knowledge of generalized Riemann-Stieltjes integral. Let $a,b\in \mathbb{R}$ and $a<b$, denote by $L^p([a,b]), p\ge1$ the usual space of Lebesgue measurable functions $f:[a,b]\to \mathbb{R}$ for which $\|f\|_{L^p} <\infty$. Let $f\in L^1(a,b)$ and $\alpha>0$. The fractional left-sided and right-sided Riemann-Liouville integral of $f$ with order $\alpha$ are defined by
	$$\begin{gathered}
	I_{a+}^\alpha f(t)=\frac{1}{\Gamma(\alpha)} \int_a^t \frac{f(s)}{(t-s)^{1-\alpha}} d s, \\
	I_{b-}^\alpha f(t)=\frac{(-1)^{-\alpha}}{\Gamma(\alpha)} \int_t^b \frac{f(s)}{(s-t)^{1-\alpha}} d s,
	\end{gathered}$$
	for all $t\in(a,b)$, where $(-1)^\alpha=e^{-i \pi \alpha}$ and $\Gamma(\alpha)=\int_0^{\infty} r^{\alpha-1} e^{-r} d r$ is the Euler Gamma function. Let $I_{a+}^\alpha\left(L^p\right)$ (resp. $I_{b-}^\alpha(L^p)$) be the image of $L^p$ by the operator $I_{a+}^\alpha$ (resp. $I_{b-}^\alpha$). If $f\in I_{a+}^\alpha\left(L^p\right)$ (resp. $I_{b-}^\alpha(L^p)$) and $0<\alpha<1$, then the Weyl derivatives of $f$ are defined as follows:
	$$
	D_{a+}^\alpha f(t):=\frac{1}{\Gamma(1-\alpha)}\left(\frac{f(t)}{(t-a)^\alpha}+\alpha \int_a^t \frac{f(t)-f(s)}{(t-s)^{\alpha+1}} d s\right) \mathbf{1}_{(a, b)}(t)
	$$
	and 
	$$
	D_{b-}^\alpha f(t):=\frac{(-1)^\alpha}{\Gamma(1-\alpha)}\left(\frac{f(t)}{(b-t)^\alpha}+\alpha \int_t^b \frac{f(t)-f(s)}{(s-t)^{\alpha+1}} d s\right) \mathbf{1}_{(a, b)}(t)
	$$
	 for almost $t\in(a,b)$ (the convergence of the integrals at the
	singularity $s = t$ holds pointwise for almost all $t \in(a, b) $ if $p = 1$ and moreover in
	$L^p$-sense if $1 < p < \infty$).
	The following two facts holds:
	{\rm (i)} If $\alpha<1/p$ and $q=\frac{p}{1-\alpha p}$, then
	$$
	I_{a+}^\alpha\left(L^p\right)=I_{b-}^\alpha\left(L^p\right) \subset L^q(a, b).
	$$ 
	{\rm (ii)} If $\alpha>1/p$, then
	$$
	I_{a+}^\alpha\left(L^p\right) \cup I_{b-}^\alpha\left(L^p\right) \subset C^{\alpha-\frac{1}{p}}(a, b) .
	$$
	The fractional integrals and derivatives are related by the inversion formulas
	$$
	\begin{array}{ll}
	I_{a+}^\alpha\left(D_{a+}^\alpha f\right)=f, &\quad  f \in I_{a+}^\alpha\left(L^p\right), \\
	D_{a+}^\alpha\left(I_{a+}^\alpha f\right)=f, &\quad  f \in L^1(a, b)
	\end{array}
	$$
	and similar ones for $I_{b-}^\alpha$ and $D_{b-}^\alpha$.
	Denote that $f(a+):=\lim _{\varepsilon \searrow 0} f(a+\varepsilon)$ and $g(b-):=\lim _{\varepsilon \backslash 0} g(b-\varepsilon)$. We define
	\begin{equation*}
	\begin{aligned}
	&f_{a+}(x):=(f(x)-f(a+)) \mathbf{1}_{(a, b)}(x), \\
	&g_{b-}(x):=(g(x)-g(b-)) \mathbf{1}_{(a, b)}(x) .
	\end{aligned}
	\end{equation*}
	We now recall the definition of generalized Riemann-Stieltjes fractional
	integral with respect to irregular functions.

\begin{defn}
Let $f$ and $g$ be functions such that the limits
$f (a+)$, $g(a+)$, $g(b-)$ exist.
Suppose that $f_{a+} \in I_{a+}^\alpha\left(L^p\right)$ and
	$g_{b-} \in  I_{b-}^\alpha\left(L^p\right)$ for some $\alpha \in (0, 1)$ and $p, q \in [1,\infty) $ such that $1/p + 1/q \le 1$. 
In this
		case the generalized Riemann-Stieltjes integral
		\begin{eqnarray*}
		\int_a^b f d g=(-1)^\alpha \int_a^b D_{a+}^\alpha f_{a+}(x) D_{b-}^{1-\alpha} g_{b-}(x) d x+f(a+)(g(b-)-g(a+)),
		\end{eqnarray*}
		is well-defined.
	\end{defn}
{
In the above explanation, $f$ and $g$ are scalar-valued just for simplicity 
of notation. Obviously, analogous results holds in the multi-dimensional 
setting.
}

	We now introduce some necessary spaces and norm. 
{In what follows we work on $[0,T]$.  
For the rest of this subsection $k, l \in \mathbb{N}$ and $0<\alpha <1/2$.}
Denote by $W_0^{\alpha,\infty} = W_0^{\alpha,\infty} ([0,T], \mathbb{R}^{k})$ 
the space of measurable functions $f:[0,T]\to \mathbb{R}^{k}$ such that
	$$\|f\|_{\alpha, \infty}:=\sup _{t \in[0, T]}\|f\|_{\alpha,[0,t]}<\infty,	$$
	where we set
	$$\|f\|_{\alpha,[0,t]}=|f(t)|+\int_0^t \frac{|f(t)-f(s)|}{(t-s)^{\alpha+1}} d s.$$
	Then, introduce H\"older continuous path space. For $\eta \in(0,1]$,  let $C^\eta = C^\eta  ([0,T], \mathbb{R}^{k})$ be the space of $\eta$-H\"older continuous functions $f:[0,T]\to \mathbb{R}^{k}$, equipped with the norm
	$$
	\|f\|_{\eta\textrm{-hld}}:=\|f\|_{\infty}+\sup _{0 \leq s<t \le T} \frac{|f(t)-f(s)|}{(t-s)^\eta}<\infty,
	$$
	with $\|f\|_{\infty}=\sup _{t \in[0, T]}|f(t)|$.
	   For any $\kappa \in (0,\alpha)$, the continuous inclusion $C^{\alpha+\kappa} \subset W_0^{\alpha, \infty} \subset C^{\alpha-\kappa}$ holds.
	
Denote by $W_0^{\alpha, 1} = W_0^{\alpha, 1} ([0,T], \mathbb{R}^{k})$ 
the space of measurable functions $f:[0,T]\to \mathbb{R}^{k}$ such that
	$$
	\|f\|_{\alpha, 1}:=\int_0^T \frac{|f(s)|}{s^\alpha} d s+\int_0^T \int_0^s \frac{|f(s)-f(y)|}{(s-y)^{\alpha+1}} d y d s<\infty
	$$
	and by $W_T^{1-\alpha, \infty} =W_T^{1-\alpha, \infty} ([0,T], \mathbb{R}^{k})$ the space of measurable functions $g:[0,T]\to \mathbb{R}^{k}$ such that
	$$
	\|g\|_{1-\alpha, \infty, T}:=\sup _{0\le s<t\le T}\left(\frac{|g(t)-g(s)|}{(t-s)^{1-\alpha}}+\int_s^t \frac{|g(y)-g(s)|}{(y-s)^{2-\alpha}} d y\right)<\infty.
	$$
	It is also easy to verify that for any $\kappa \in (0,\alpha)$, $C^{1-\alpha+\kappa} \subset W_T^{1-\alpha, \infty} \subset C^{1-\alpha-\kappa}$.
	For $g \in W_T^{1-\alpha, \infty}$, we have
	$$
	\begin{aligned}
	\Lambda_\alpha(g) :=\frac{1}{\Gamma(1-\alpha)} \sup _{0<s<t<T}\mid\left(D_{t-\alpha}^{1-\alpha} g_{t-}\right)(s) \mid \leq \frac{1}{\Gamma(1-\alpha) \Gamma(\alpha)}\|g\|_{1-\alpha, \infty, T}<\infty.
	\end{aligned}
	$$
	Moreover, if $f \in W_0^{\alpha, 1}([0,T], \mathbb{R}^{l\times k})$ 
	 and $g \in W_T^{1-\alpha, \infty}([0,T], \mathbb{R}^{k})$, the integral $\int_0^t f dg \in \mathbb{R}^{l}$ is well-defined and the estimate
	$$
	\left|\int_0^t f d g\right| \leq \Lambda_\alpha(g)\|f\|_{\alpha, 1}
	$$
	holds for all $t \in [0,T]$.


		\subsection{Fractional Brownian motion}\label{sec-2-2}
	{
		We introduce a mixed fractional Brownian motion
		of Hurst parameter $H$ and recall some basic facts on it for later use.
		   In this subsection, $H \in (0,1)$ unless otherwise stated. 
		In what follows, 
		   $d_1, d_2 \in \mathbb{N}$ and  $d:=d_1+ d_2$.
		   		}
		
		Consider an $d_1$-dimensional fractional Brownian motion
		(fBm for short) 
		\[
		(B^H_t)_{t\in [0,T]}=(B_t^{H,1},B_t^{H,2},\cdots,B_t^{H,d_1})_{t\in [0,T]}
		\]
		 with Hurst parameter $H\in (0,1)$,
		where $(B_t^{H,i})_{t\in [0,T]}$, $1 \le i\le d_1$,
		are independent one-dimensional fBm's.
		This is a centered Gaussian process, characterized by  the covariance formula
		\[
		\mathbb{E}\big[B_{t}^{H,i} B_{s}^{H,j}\big]
		   =\frac{\delta_{ij}}{2}\left[t^{2 H}+s^{2 H}-|t-s|^{2 H}\right], 
		    \qquad  s, t \in [0,T], \, 1 \le i, j\le d_1.
		\]
		Here, $\delta_{ij}$ stands for Kronecker's delta.
				The increment satisfies that 
		\[
		\mathbb{E}\big[(B_{t}^{H,i}-B_{s}^{H,i})(B_{t}^{H,j}-B_{s}^{H,j})\big]=\delta_{ij}|t-s|^{2 H}, 
		\qquad   s, t \in [0,T], \, 1 \le i, j\le d_1.
		\]
		When $H=1/2$, it is a standard $d_1$-dimensional Brownian motion
		(Bm for short).

A typical construction for fBm is as follows (see e.g. \cite{1999Decreusefond}).
For a standard $d_1$-dimensional Brownian motion 
$B=(B_t)_{t\in [0,T]}$, define $B^H=(B_t^H)_{t\in [0,T]}$ by
		\begin{equation}\label{vol-rep}
		B_t^H :=\int_0^T K_H(t, s) d B_s, \qquad t \in[0,T].
		\end{equation}
		Here, we set for all $0\le s \le t \le T$
				$$
		K_H(t, s) :=k_H(t, s) 1_{[0, t]}(s) 
				$$
	with
		$$
		k_H(t, s) :=\frac{c_H}{\Gamma\left(H+\frac{1}{2}\right)}(t-s)^{H-\frac{1}{2}} F\left(H-\frac{1}{2}, \frac{1}{2}-H, H+\frac{1}{2} ; 1-\frac{t}{s}\right),
		$$
		where
		$c_H=\left[\frac{2 H \Gamma\left(\frac{3}{2}-H\right) \Gamma\left(H+\frac{1}{2}\right)}{\Gamma(2-2 H)}\right]^{1 / 2}$,  $\Gamma$ is a gamma function and $F$ is the Gauss hypergeometric function. 
		(For more details, see \cite{2020Budhiraja}.)
		It is known that the map $B \mapsto B^H$ is (essentially) 
		 a measurable isomorphism from the classical Wiener space to
		 the abstract Wiener space of fBm.
		   In other words, the two Gaussian structures can be 
		   identified through this map.
		     Moreover, the augmented natural filtration
		     of  $B$ and that of $B^H$ coincide.

By the above increment formula and 
Kolmogorov's continuity criterion, 
 $B^H$ has  almost surely $H'$-H\"older continuous trajectories for every $H'\in(0,H)$. Moreover,  we have the following remark when $H >1/2$.
	
	\begin{rem}
		Let $1/2<H<1$ and $1- H < \alpha < 1/2 $. Then, the trajectories of fBm $(B^H_t)_{t\in [0,T]}$
				belong to the space $W_T^{1-\alpha, \infty} ([0,T], \mathbb{R}^{d_1})$. 
		Therefore, the generalized Riemann-Stieltjes
		integrals $\int_0^t v_s d B_s^H$ exists if ${(v_t)_{t \in [0, T ]}}$ is a stochastic process whose trajectories belong to the space 
		$W_0^{\alpha, 1}([0,T], \mathbb{R}^{m\times d_1})$. 
		  Moreover, it holds that
		$$
		\left|\int_0^t v_s d B_s^H\right| \leq \Lambda_\alpha\left(B^H\right)\|v\|_{\alpha, 1}.
		$$
		Here,  $\Lambda_\alpha\left(B^H\right):=\frac{1}{\Gamma(1-\alpha) \Gamma(\alpha)}\left\|B^H\right\|_{1-\alpha, \infty, T}$, which has moments of all order \cite[Lemma 7.5]{2002Rascanu}.
	\end{rem}

For $f\in L^2([0,T],\mathbb{R}^{d_1})$, we define
		\[
		\mathcal{K_H}f(t)=\int_0^T K_H(t, s)f(s) d s, \qquad t \in[0,T].
		\]	
	The Cameron-Martin Hilbert space 
	$\mathcal{H}^{H,d_1}= \mathcal{H}^{H}([0,T], \mathbb{R}^{d_1})$
	for $(B^H_t)_{0\le t\le T}$ is  defined by
	\[
	\mathcal{H}^{H,d_1}=\big\{\mathcal{K_H} \dot{h}: \dot{h}\in L^2\left([0,T], \mathbb{R}^{d_1}\right)\big\}
	\]
(Note that $\dot h$ is {\it not} the time derivative of $h$.) 
It should be recalled that $\mathcal{H}^{H,d_1} \subset C^{H'}([0,T], \mathbb{R}^{d_1})$ for all $H' \in (0,H)$.
The scalar inner product on  $\mathcal{H}^{H,d_1}$ is defined by
	\[
	\langle h, g\rangle_{\mathcal{H}^{H,d_1}}=\langle \mathcal{K_H} \dot{h},\mathcal{K_H} \dot{g}\rangle_{\mathcal{H}^{H,d_1}}:=\langle\dot{h}, \dot{g}\rangle_{L^2}.
	\]
	Thus, $\mathcal{K_H}:L^2([0,T],\mathbb{R}^{d_1}) \to {\mathcal{H}^{H,d_1}}$ is a unitary isometry. For more details, see {\cite{1999Decreusefond}}.

	We also 	consider a standard ${d_2}$-dimensional  Brownian motion $(W_t)_{t\in [0,T]}$.
	 Throughout, $(W_t)$ and $(B^H_t)$ are assumed to be independent.
	 The $d$-dimensional 
	 process $(B^H_t ,W_t)_{t\in [0,T]}$ is called  mixed fBm 
	 of Hurst parameter $H$.
	The Cameron-Martin Hilbert space $\mathcal{H}^{\frac{1}{2},d_2} =\mathcal{H}^{\frac{1}{2}}([0,T], \mathbb{R}^{d_1})$ 
	for  $(W_t)_{t\in  [0,T]}$  is defined by
	\[
	\mathcal{H}^{\frac{1}{2},d_2}:=\big\{  v \in C([0,T], \mathbb{R}^{d_2}) :  v _{t}=\int_{0}^{t}  v^{\prime} _{s} d s \text { for all } t\in[0,T] \text { for some  }v^{\prime} \in L^2([0,T],\mathbb{R}^{d_2})\big\}.
	\]
We define $\langle v, w \rangle_{\mathcal{H}^{\frac{1}{2},d_2}}^{2}:=
\int_{0}^{T} \langle v^{\prime} _{t}, w^{\prime} _{t}\rangle_{\mathbb{R}^{d_2}}^{2} d t<\infty$.
	Then $\mathcal{H}:={{\mathcal{H}^{H,d_1}}\oplus{\mathcal{H}^{\frac{1}{2},d_2}}}$ is the Cameron-Martin subspace for  mixed fBm $(B_t^H, W_t)_{0\le t\le T}$. 
	For $0<N<\infty$, we set 
	\[
	S_N=\left\{(u,v) \in \mathcal{H}: 
	\frac{1}{2}\|(u,v)\|_{\mathcal{H}}^2 := \frac{1}{2} 
	(\|u\|_{\mathcal{H}^{H,d_1}}^{2} +\|v \|_{\mathcal{H}^{\frac{1}{2},d_2}}^{2})
	\leq N\right\}. 
	\]
	Equipped with the weak topology, the ball $S_N$ can
	be metrized as a compact Polish space.

Let 	$(\Omega, \mathcal{F}, \mathbb{P})$ be the $d$-dimensional 
	 classical Wiener space, namely,  
	  {\rm (i)} $\Omega=C_0 ([0,T], \mathbb{R}^d )$
	     is the space of $\mathbb{R}^d$-valued continuous paths 
	       starting at $0$ with the uniform topology, 
	           {\rm (ii)} $\mathbb{P}$ is
	              the $d$-dimensional Wiener measure and 
	          {\rm (iii)} $\mathcal{F}$ is the $\mathbb{P}$-completion of
	          the Borel $\sigma$-field on $\Omega$.             
The coordinate process is denoted by $(B_t, W_t)_{t\in [0,T]}$,
  which is of course a standard
     $d$-dimensional Bm under $\mathbb{P}$. 
	  Denote by $\{\mathcal{F}_t\}_ {t \in[0,T]}$   
	   the natural augmented filtration, that is,
	     $\mathcal{F}_t :=\sigma \{(B_s,W_s): 0 \leq s \leq t\} \vee \mathcal{N}$ and $\mathcal{N}$ is the set of all $\mathbb{P}$-negligible events. 
	     As is well-known, $\{\mathcal{F}_t\}$ satisfies the usual condition.
		        Construct $B^H$ from $B$ by using \eqref{vol-rep}.
		             Then, $(B^H_t, W_t)_{t\in [0,T]}$ is mixed fBm 
		               with Hurst parameter $H$ defined on $(\Omega, \mathcal{F}, \mathbb{P})$.
		        It is known that $\mathcal{F}_t=\sigma \{(B^H_s,W_s): 0 \leq s \leq t\} \vee \mathcal{N}$ for all $t$.     
		      For the rest of this paper, we work on 
		         $(\Omega, \mathcal{F}, \mathbb{P})$ and 
		         our mixed fBm $(B^H, W)$ is  realized in the above way.

We denote by $\mathcal{B}_b^N$, $0<N<\infty$, 
the set of all $\mathbb{R}^d$-valued  $\{\mathcal{F}_t\}$-progressively measurable processes $(\phi_t, \psi_t)_{t \in [0,T]}$ on 
		         $(\Omega, \mathcal{F}, \mathbb{P})$ satisfying that 
\[
\frac12 \int_0^T  ( |\phi_t|^2_{\mathbb{R}^{d_1}}+ |\psi_t|^2_{\mathbb{R}^{d_2}} )  dt  \le N,
\qquad 
\mbox{$\mathbb{P}$-a.s.}
\]
and set $\mathcal{B}_b =\cup_{0<N<\infty} \mathcal{B}_b^N$.
    For every $(\phi, \psi) \in \mathcal{B}_b$, the law of the process
      $(B + \int_0^\cdot  \phi_s ds, W+ \int_0^\cdot  \psi_s ds)$ is 
        mutually absolutely continuous to 
            that of $(B, W)$, which is $\mathbb{P}$.
              This can be easily verified  since Girsanov's theorem 
              and Novikov's criterion are available
                 in this case.

We denote by $\mathcal{A}_b^N$, $0<N<\infty$, the set of all 
	$\mathbb{R}^d$-valued $\{\mathcal{F}_t\}$-progressively measurable  processes $(u_t, v_t)_{t \in [0,T]}$ on 
		         $(\Omega, \mathcal{F}, \mathbb{P})$ of the form  	
\[
(u, v)=\Bigl( \mathcal{K}_H \dot{u}, \, \int_0^{\cdot} v^{\prime}_s ds \Bigr),
\qquad 
\mbox{$\mathbb{P}$-a.s.}\quad
\mbox{for some $(\dot{u}, v^{\prime}) \in \mathcal{B}_b^N$.}
\]	
We set $\mathcal{A}_b =\cup_{0<N<\infty} \mathcal{A}_b^N$.
Every $(u, v) \in  \mathcal{A}_b^N$ can be viewed as an $S_N$-valued random 
 variable.
		Since $S_N$ is compact, $\{ \mathbb{P}\circ (u, v)^{-1} :  
		  (u, v) \in \mathcal{A}_b^N\}$ is automatically tight.	
			 {It is easy to see that, for every $(u, v) \in \mathcal{A}_b$}, 
			 the law of the process $(B^H +u, W+v)$ is 
        mutually absolutely continuous to  that of $(B^H, W)$.

A variational representation formula for  mixed fBm is now given.
		\begin{prop}\label{prop1}
			Let $\Phi \colon \Omega \to \mathbb{R}$ be a bounded Borel measurable function. Then, we have
			\begin{eqnarray*}\label{2-1}
			- \log \mathbb{E}\Big[\exp \big(-{\Phi ( B^H, W)}\big)\Big] 
			=\inf _{(u,v) \in \mathcal{A}_b} \mathbb{E}\Big[\Phi( B^H+u, W+v)+\frac{1}{2}\|(u,v)\|_{\mathcal{H}}^2\Big].
			\end{eqnarray*}
	\end{prop}
	
\begin{proof}
First, let us recall Bou\'e-Dupuis' variational representation formula for
 a standard Bm.  It states that, for any bounded Borel measurable function $F$ on $\Omega$, we have  
  {\begin{eqnarray*}\label{2-1}
 		- \log \mathbb{E}\Big[\exp \big(-{F( B, W)}\big)\Big] 
 		=\inf _{(\phi, \psi) \in \mathcal{B}_b} 
 		\mathbb{E}\Big[F( B+\int_0^\cdot\phi_s ds, W+\int_0^\cdot\psi_s ds)+\frac{1}{2}\|(\phi, \psi)\|_{L^2 ([0,T], \mathbb{R}^d)}^2\Big].
 \end{eqnarray*}}
See \cite[Theorems 3.14 and 3.17]{BP_book} for example.

In this formula,
  it should be noted that $F$ can be a {\it $\mathbb{P}$-equivalence class} of 
   bounded measurable functions (instead of an {\it everywhere-defined} one).
  	 That is, if $F = \hat{F}$, $\mathbb{P}$-a.s. (i.e.
	  $F (B,W) = \hat{F} (B,W)$, $\mathbb{P}$-a.s.),  
	    then we have $F(B + \int_0^\cdot  \phi_s ds, W+ \int_0^\cdot  \psi_s ds) = \hat{F} (B + \int_0^\cdot  \phi_s ds, W+ \int_0^\cdot  \psi_s ds) $,  $\mathbb{P}$-a.s. for any $(\phi, \psi) \in \mathcal{B}_b$.
	    	    Here, we used the mutual absolute continuity mentioned above.
		    
Finally, by letting $F$ be the measurable map $(B, W) \mapsto \Phi (B^H, W)$
  (and replacing the symbol $(\phi, \psi)$ by $(\dot{u}, v^\prime)$), 
   we obtain the desired formula.  This completes the proof.
	 	\end{proof}

\begin{rem}
In preceding works \cite{2020Budhiraja, 2022Gailus}
on large deviations for  SDEs driven by (mixed) fBm, 
X. Zhang's generalized variational representation formula
\cite[Theorem 3.2]{2009Zhang} on abstract Wiener spaces is used. 
 However, we avoid it because this generalized formula is based on a
  deep theory of transformations on abstract Wiener spaces
    and therefore looks a little bit too heavy for our purpose.
   	\end{rem}


	\section{Assumptions and Statement of Main Result}\label{sec-2-3}
	In this section, we first introduce assumptions of our main theorem and then state our main theorem.
{
In what follows, $1/2<H< 1$ will always be assumed. 
In the slow-fast system (\ref{1}), the small parameters satisfy 
$0 < \delta <\varepsilon \le 1$. 
Later we will assume $\delta =o (\varepsilon)$ and 
let $\varepsilon \to 0$.
	}

	To ensure the existence and uniqueness of solutions to the system (\ref{1}) we assume:
	\begin{itemize}
		\item[(\textbf{A1}).] The function $\sigma_1$ is of class of $C^1$. There exists a  constant $L> 0$ such that for  any $ x_1 , x_2\in \mathbb{R}^{m}$, 
		$$\left|\nabla \sigma_1\left( x_1\right)\right| \leq L,	\qquad
		\left|\nabla \sigma_1\left( x_1\right)-\nabla \sigma_1\left( x_2\right)\right| \leq L\left|x_1-x_2\right|$$
			hold. Here, $\nabla$ is the standard gradient operator on $\mathbb{R}^{m}$.
		\item[(\textbf{A2}).] There exists a constant $L> 0$ such that for  any $ (x_1,y_1) $,  $ (x_2,y_2)\in \mathbb{R}^{m} \times\mathbb{R}^{n}$, 
		\begin{equation*}
		\begin{aligned}
		&\left|f_1\left( x_1, y_1\right)-f_1\left( x_2, y_2\right)\right| +
		\left|f_2\left(x_1, y_1\right)-f_2\left(x_2, y_2\right)\right|
		\\
		&\qquad \qquad 
		+\left|\sigma_2\left(x_1, y_1\right)-\sigma_2\left(x_2, y_2\right)\right|  \leq L\left(\left|x_1-x_2\right|+\left|y_1-y_2\right|\right), 
		\end{aligned}
		\end{equation*}	
		and
	 {	\begin{equation*}
		\sup_{y_1\in\mathbb{R}^{n}}(\left|f_1\left(x_1, y_1\right)\right| +\left|\sigma_2\left(x_1, y_1\right)\right| ) \leq L\left(1+\left|x_1\right|\right)
		\end{equation*}}
			hold.
	\end{itemize}	
Clearly, (\textbf{A2}) implies that there exists a constant $L^{\prime}>0$ such that for any $(x,y)\in \mathbb{R}^m \times \mathbb{R}^n$, 	$$
\left|f_2\left(x, y\right)\right| \leq L^{\prime}\left(1+\left|x\right|+\left|y\right|\right)$$
holds.

	Under Assumptions (\textbf{A1}) and (\textbf{A2}) above,  one can deduce from \cite[Theorem 2.2]{Guerra} that the slow-fast system (\ref{1}) admits a unique (pathwise) solution $(x^{(\varepsilon, \delta)},y^{(\varepsilon,\delta)})$. Then there is a measurable map 
	\[
	\mathcal{G}^{(\varepsilon, \delta)} (\sqrt{\varepsilon}\bullet,\sqrt{\varepsilon}\star): C_0\left([0,T], \mathbb{R}^d\right) \rightarrow C^{\alpha}\left([0,T], \mathbb{R}^m\right)
	\]
 {with $1-H<\alpha<1/2$} such that
	$x^{(\varepsilon,\delta)}:=\mathcal{G}^{\varepsilon,\delta}(\sqrt \varepsilon B^H, \sqrt \varepsilon W)$. 
{In other words, this is 
	a Borel measurable version of the slow component of 
	 the solution map of SDE (\ref{1}).
	}

In order to study an averaging principle and 
a large deviation principle for the system (\ref{1}),
we further assume: 	
\begin{itemize}
		\item[(\textbf{A3}).]  There exist positive constants $C>0$ and $\beta_i >0 \,(i=1,2)$ such that for any  $ (x,y_1),(x,y_2)\in \mathbb{R}^{m} \times\mathbb{R}^{n}$
	\begin{equation*}
	\begin{aligned}
	2\left\langle y_1-y_2, f_2\left(x, y_1\right)-f_2\left(x, y_2\right)\right\rangle+\left|\sigma_2\left(x, y_1\right)-\sigma_2\left(x, y_2\right)\right|^2 
	 &\leq-\beta_1\left|y_1-y_2\right|^2, \\
	2\left\langle y_1, f_2\left(x, y_1\right)\right\rangle+\left|\sigma_2\left(x, y_1\right)\right|^2 & \leq-\beta_2\left|y_1\right|^2+C|x|^2+C
	\end{aligned}
	\end{equation*}
	hold.
	\end{itemize}		
{
We remark that Assumptions (\textbf{A1})--(\textbf{A3})  are 
slightly stronger than their counterparts in \cite{2023Pei}.
So, we can use results in \cite{2023Pei}.	
}
Assumption  (\textbf{A3}) ensures that the solution to the following equation with frozen ${x}\in\mathbb{R}^{m}$
\[
d\tilde{y}_t = f_{2}({x}, \tilde{y}_t) dt +\sigma_2( {x}, \tilde{y}_t)dW_{t}
\]
has a unique invariant probability measure $\mu_{{x}}$, which can be deduced from \cite[Theorem 6.3.2]{DaPratoZabczyk}. 

	The skeleton equation is defined  as follows
\begin{eqnarray}\label{3}
d\tilde{x}_t = \bar{f}_1(\tilde{x}_t)dt + \sigma_{1}( \tilde{x}_t)du_t,\quad\tilde{x}_0=x_0
\end{eqnarray}	
with $\bar{f}_1(x)=\int_{\mathbb{R}^{d}}f_{1}(x, { y})\mu_{x}(d{ y})$. 
By similar estimates to \cite[Lemma 3.2]{2023Pei},  
{
$\bar{f}_1$ is   Lipschitz continuous and linear growth. 
}
For $(u, v) \in S_{N}$,  there exists a unique solution $\tilde {x} \in  W_0^{\alpha, \infty}([0,T], \mathbb{R}^{m})$  to the skeleton equation (\ref{3}), see  \cite[Proposition 3.6]{2020Budhiraja}. Moreover, it holds that
$$\|\tilde{x}\|_{1-\alpha}\le c,$$
for some constant $c=c_N$ independent of $(u, v)\in S_N$. Since $1-\alpha>\alpha$, we have the following compact embedding $C^{1-\alpha}([0,T],\mathbb{R}^{m}) \subset C^{\alpha}([0,T],\mathbb{R}^{m})$.

We  also define a map
	\[
	\mathcal{G}^{0}: \mathcal{H} \rightarrow  C^{\alpha}\left([0,T], \mathbb{R}^m\right)
	\]
	by  $\tilde{x}=\mathcal{G}^{0}(u, v)$. In other words, this is the solution map of the above skeleton equation. Note that $\mathcal{G}^{0}(u, v)$ is independent	of $v$.

	
	Now, we provide a precise statement of our main theorem.
	 {
	 We will show this theorem in Section \ref{sec.5}.	
	 } 
	  	\begin{thm}\label{thm}
	Let $H\in(1/2,1)$  and $1-H<\alpha<1/2$ and assume (\textbf{A1})--(\textbf{A3}).
	{Assuming $\delta =o (\varepsilon)$, we let $\varepsilon \to 0$.
	}
		  Then, the slow variable $x^{(\varepsilon,\delta)}$ of the system (\ref{1}) satisfies a large deviation principle on the  H\"older path space 
	$C^{\alpha}\left([0,T], \mathbb{R}^m\right)$ with the good rate function $I: C^{\alpha}\left([0,T], \mathbb{R}^m\right)\rightarrow [0, \infty]$   defined by
		\begin{eqnarray*}\label{rate}
		I(\xi) &=& {
		\inf\Big\{\frac{1}{2}\|u\|^2_{\mathcal{H}^{H,d_1}}~:~{ u\in \mathcal{H}^{H,d_1} 
		\quad\text{such that} \quad\xi =\mathcal{G}^{0}(u, 0)}\Big\} 
		}
		\cr
		&=& \inf\Big\{\frac{1}{2}\|(u,v)\|^2_{\mathcal{H}} ~:~ {( u,v)\in \mathcal{H}\quad\text{such that} \quad\xi =\mathcal{G}^{0}(u, v)}\Big\},
		\qquad 
		    \xi\in C^{\alpha}\left([0,T], \mathbb{R}^m\right).		\end{eqnarray*}
	\end{thm}


{
\begin{rem}\label{rem1}
We make a comment on the topology of the path space. 
At the moment 
there are two works in which the weak convergence method 
is applied to stochastic differential equations driven by  (mixed) fBm
(see \cite{2020Budhiraja, 2022Gailus}).
In these works, the topology is the uniform one on the path space.
In our main theorem above, however,
the topology is the $\alpha$-H\"older topology ($1-H < \alpha <1/2$).
This is a slight improvement.
%
%
\end{rem}
}
\begin{rem} 	
	 {The variational formula  cannot be used on the   space $ C^\alpha ([0,T], \mathbb{R}^{k} )$ directly since  it is  not seperable and not a  Polish space. Then we introduce   a   Banach subspace $H^\alpha = H^\alpha  ([0,T], \mathbb{R}^{k})$, which  is the space that for all $g\in C^\alpha([0,T], \mathbb{R}^{k})$,  equipped with the norm
		\begin{eqnarray*}\label{111}
			\lim_{\delta\to0+}\sup_{\substack{|t-s|\le \delta\\0 \leq s<t \le T}} \frac{|g(t)-g(s)|}{(t-s)^\alpha}=0. 
		\end{eqnarray*}
The space $H^\alpha$ is a separable space and 
		$$
		H^\alpha=\overline{\bigcup_{\kappa>0}C^{\alpha+\kappa}},
		$$
		where the closure is taken in  the norm $\|\cdot\|_{\alpha\textrm{-hld}}$. The above space $H^\alpha$ is continuously embedded in the usual $\alpha$-H\"older space. For more information, see  \cite{Ciesielski}.}
	
	 {For any given $\alpha$ that satisfies the condition on the H\"older exponent,
		we could find slightly large $\beta$ such that it still satisfies the condition.
		Then, if our process takes values in  $ C^\beta ([0,T], \mathbb{R}^{k} )$,  it is not too difficult to see that it also belongs to the space $ H^\alpha$.
		Then we  could apply   variational representation 
			on the space $H^\alpha$  and only need to  prove the weak convergence method under the $\alpha$-H\"older norm.
		Finally,  the same LDP holds on the usual $\alpha$-H\"older space by employing the contraction principle \cite[Theorem 4.2.1]{Dembo2009Large} directly. }
\end{rem}

{
\begin{rem}\label{rem2}
%
In \cite{2022Gailus}, Gailus and Gasteratos established the large deviation for a slow-fast system driven by  mixed fBm. 
We remark that there exist some differences between \cite{2022Gailus} and this work (except the topology issue in the above remark): 
(1) For the slow-fast system in \cite{2022Gailus}, the diffusion term in the   fast system is independent of slow component, but in our system \eqref{1}, the fast component fully depends on the slow component. (2) In \cite{2022Gailus}, they applied homogenization theory and viable pair method. But in this work, we  employ the method that is a combination  of     weak convergence criteria and Khasminskii averaging principle technique. An advantage of this approach is that the required conditions are easy to be verified. 
%
\end{rem}
}


	\section{Preliminary Lemmas}
	
	In this Section, $H$ and $\alpha$ are such that $H\in(1/2,1)$ and  $1-H<\alpha<1/2$. We assume $0<\delta<\varepsilon\le 1$.
	Those are basically fixed. 
	
	To prove \thmref{thm}, some prior estimates should be given. 
	Firstly, for a pair of control $(u^{(\varepsilon, \delta)}, v^{(\varepsilon, \delta)})\in \mathcal{A}^{b}$, consider the following controlled system associated to (\ref{1}):
	\begin{eqnarray}\label{2}
	\left
	\{
	\begin{array}{ll}
	d\tilde {x}^{(\varepsilon, \delta)}_t =& f_{1}(\tilde {x}^{(\varepsilon, \delta)}_t, \tilde {y}^{(\varepsilon, \delta)}_t)dt + \sigma_{1}(\tilde {x}^{(\varepsilon, \delta)}_t)du^{(\varepsilon, \delta)}_t+ \sqrt \varepsilon  \sigma_{1}(\tilde {x}^{(\varepsilon, \delta)}_t)dB^H_{t}\\
	\delta d\tilde {y}^{(\varepsilon, \delta)}_t =&  f_{2}( \tilde {x}^{(\varepsilon, \delta)}_t, \tilde {y}^{(\varepsilon, \delta)}_t)dt +\sqrt{\frac{ \delta}{\varepsilon}}\sigma_2(\tilde {x}^{(\varepsilon, \delta)}_t, \tilde {y}^{(\varepsilon, \delta)}_t)dv^{(\varepsilon, \delta)}_t+ {\sqrt \delta }\sigma_2(\tilde {x}^{(\varepsilon, \delta)}_t, \tilde {y}^{(\varepsilon, \delta)}_t)dW_{t},
	\end{array}
	\right.
	\end{eqnarray}
	with initial value $(\tilde x^{(\varepsilon, \delta)}_0, \tilde y^{(\varepsilon, \delta)}_0)=(x_0, y_0)\in \mathbb{R}^{m}\times \mathbb{R}^{n}$. It is not hard to verify that there exists a unique solution $(\tilde {x}^{(\varepsilon,\delta)}, \tilde {y}^{(\varepsilon,\delta)})=\mathcal{G}^{(\varepsilon, \delta)} (\sqrt{\varepsilon}B^H+u^{(\varepsilon,\delta)},\sqrt{\varepsilon}W+v^{(\varepsilon,\delta)})$ of the system (\ref{2}). 
	
	\begin{lem}\label{lem1}
		Assume (\textbf{A1})--(\textbf{A3}) and let  $p\ge1$ and  $N\in(0,\infty)$. Then, there exists $C>0$ such that for every $(u^{(\varepsilon,\delta)},v^{(\varepsilon,\delta)})\in \mathcal{A}_b^N$, we have
		\begin{equation*}
		\mathbb{E}\big[\|\tilde x^{(\varepsilon,\delta)}\|_{\alpha, \infty}^p\big] \leq C.
		\end{equation*}
		Here, $C$ is a positive constant which depends only on $p$ and $N$.
	\end{lem}
\para{Proof}. {In this Proof, $C$ is a positive constant which depends only on $p$ and $N$ which may change from line to line.} Firstly, denote $\Lambda:=\Lambda_\alpha\left(B^H\right) \vee 1$, and for any $\lambda\ge 1$, set 
$$
\|f\|_{\lambda, t}:=\sup _{0 \leq s \leq t} e^{-\lambda s}|f(s)|
$$
and 
$$\|f\|_{1, \lambda, t}:=\sup _{0 \leq s \leq t} e^{-\lambda s} \int_0^s \frac{|f(s)-f(r)|}{(s-r)^{\alpha+1}} d r.$$ 

By Assumptions (\textbf{A1})--(\textbf{A2}) and the fact that $
\left|\int_0^t v_s d B_s^H\right| \leq \Lambda_\alpha\left(B^H\right)\|v\|_{\alpha, 1}
$, it deduces that
\begin{equation*}\label{3-2}
\begin{aligned}
\|\tilde x^{(\varepsilon,\delta)}\|_{\lambda, t} &=\sup _{0 \leq s \leq t} e^{-\lambda s}\left|x_0+\int_0^s f_1(\tilde x_r^{(\varepsilon,\delta)}, \tilde y_r^{(\varepsilon,\delta)}) d r+\int_0^s \sigma_1( \tilde x_r^{(\varepsilon,\delta)}) d u_r^\varepsilon +\sqrt{\varepsilon}\int_0^s \sigma_1( \tilde x_r^{(\varepsilon,\delta)}) d B_r^H\right| \cr
& \leq C \Lambda\big[1+\sup _{0 \leq s \leq t} \int_0^s e^{-\lambda(s-r)}\big(r^{-\alpha}\|\tilde x^{(\varepsilon,\delta)}\|_{\lambda, t}+\|\tilde x^{(\varepsilon,\delta)}\|_{1, \lambda, t}\big) d r\big] \cr
& \quad+ C\|u^\varepsilon\|_{1-\alpha,\infty} \big[1+\sup _{0 \leq s \leq t} \int_0^s e^{-\lambda(s-r)}\big(r^{-\alpha}\|\tilde x^{(\varepsilon,\delta)}\|_{\lambda, t}+\|\tilde x^{(\varepsilon,\delta)}\|_{1, \lambda, t}\big) d r\big] \cr
& \leq C \Lambda\big[1+\sup _{0 \leq s \leq t} \int_0^s e^{-\lambda(s-r)}\big(r^{-\alpha}\|\tilde x^{(\varepsilon,\delta)}\|_{\lambda, t}+\|\tilde x^{(\varepsilon,\delta)}\|_{1, \lambda, t}\big) d r\big] \cr
& \quad+ C\|u^\varepsilon\|_{\mathcal{H}^{H,d_1}} \big[1+\sup _{0 \leq s \leq t} \int_0^s e^{-\lambda(s-r)}\big(r^{-\alpha}\|\tilde x^{(\varepsilon,\delta)}\|_{\lambda, t}+\|\tilde x^{(\varepsilon,\delta)}\|_{1, \lambda, t}\big) d r\big] ,
\end{aligned}
\end{equation*}
where the second inequality comes from the fact that for a sufficiently small constant $\kappa>0$, $\|u^\varepsilon\|_{1-\alpha,\infty} \le C \|u^\varepsilon\|_{1-\alpha+\kappa}\le C\|u^\varepsilon\|_{H}\le C \|u^\varepsilon\|_{\mathcal{H}^{H,d_1}}$, see \cite[Lemma 4.1]{2020Budhiraja}.

After  using the estimate that 
\begin{eqnarray}\label{3-3}
\int_0^t e^{-\lambda(t-r)} r^{-\alpha} d r 
\leq \lambda^{\alpha-1} \sup _{z>0} \int_0^z e^{-y}(z-y)^{-\alpha} d y 
\leq C \lambda^{\alpha-1},
\end{eqnarray}
we can get 
\begin{eqnarray}\label{3-4}
\|\tilde x^{(\varepsilon,\delta)}\|_{\lambda, t} 
&\le& (C\Lambda+\|u^\varepsilon\|_{\mathcal{H}^{H,d_1}})\big(1+\lambda^{\alpha-1}\|\tilde x^{(\varepsilon,\delta)}\|_{\lambda, t}+\lambda^{-1}\|\tilde x^{(\varepsilon,\delta)}\|_{1, \lambda, t}\big).
\end{eqnarray}
It proceeds to estimate $\|\tilde x^{(\varepsilon,\delta)}\|_{1, \lambda, t}$. 
 {Firstly, we give some prior estimates of $$\mathcal{C}_1:=\sup _{0 \leq s \leq t} e^{-\lambda s} \int_0^s(s-r)^{-\alpha-1}\left|\int_r^s f(\tilde{x}_q^{(\varepsilon,\delta)},\tilde{y}_q^{(\varepsilon,\delta)}) d q\right| d r.$$
	By  (\textbf{A2}) and Fubini's theorem, we obtain
	\begin{eqnarray*}
		\mathcal{C}_1 
		&\leq & C \sup _{0 \leq s \leq t} e^{-\lambda s} \int_0^s(s-r)^{-\alpha-1} \int_r^s(1+|\tilde{x}_q^{(\varepsilon,\delta)}|) d q d r \\
		&\leq &  C\left(1+\sup _{0 \leq s \leq t} e^{-\lambda s} \int_0^s(s-r)^{-\alpha-1} \int_r^s(q-r)^\alpha \frac{|\tilde{x}_q^{(\varepsilon,\delta)}|}{(q-r)^\alpha} d q d r\right) \\
		&	\leq & C\left(1+\sup _{0 \leq s \leq t} e^{-\lambda s} \int_0^s(s-r)^{-\alpha-1} \int_r^s \frac{|\tilde{x}_q^{(\varepsilon,\delta)}|}{(q-r)^\alpha} d q d r\right) \\
		&= & C\left(1+\sup _{0 \leq s \leq t} e^{-\lambda s} \int_0^s \int_0^q(s-r)^{-\alpha-1}(q-r)^{-\alpha} d r|\tilde{x}_q^{(\varepsilon,\delta)}| d q\right) \\
		&\leq & C\left(1+ \sup _{0 \leq s \leq t} e^{-\lambda s} \int_0^s(s-q)^{-2 \alpha}|\tilde{x}_q^{(\varepsilon,\delta)}| d q\right) \\
		&\leq & C\left(1+ \sup _{0 \leq s \leq t} \int_0^s e^{-\lambda(s-q)}(s-q)^{-2 \alpha}\|\tilde{x}_q^{(\varepsilon,\delta)}\|_{\lambda, t} d q\right) .
\end{eqnarray*}
Here, the fifth line comes from the  estimate that   
\begin{eqnarray}\label{alpha}
\int_0^q(s-r)^{-\alpha-1}(q-r)^{-\alpha} d r= (s-q)^{-2 \alpha}\int_0^{q/{(s-q)}}(1+y)^{-\alpha-1}y^{-\alpha} d y
\le c_\alpha (s-q)^{-2 \alpha}
\end{eqnarray} 
where $s=r-(t-r)y$ and  $c_\alpha=\int_0^{\infty}(1+q)^{-\alpha-1} q^{-\alpha} d q=B(2 \alpha, 1-\alpha)$ with $B$ being the Beta function.}

Then, we define
\begin{eqnarray*}\label{3-5}
\mathcal{C}_2:=\int_0^t(t-s)^{-\alpha-1}\left|\int_s^t f(r) d B_r^H\right| d s.
\end{eqnarray*}
Using Fubini's theorem and some simple computation, it holds that
\begin{eqnarray*}\label{3-6}
\mathcal{C}_2 &\le & \Lambda\left(\int_0^t \int_s^t(t-s)^{-\alpha-1} \frac{|f(r)|}{(r-s)^\alpha} d r d s+\int_0^t \int_s^t \int_s^r(t-s)^{-\alpha-1} \frac{|f(r)-f(q)|}{(r-q)^{1+\alpha}} d q d r d s\right) \cr
&\le& \Lambda\left(\int_0^t \int_0^r(t-s)^{-\alpha-1}(r-s)^{-\alpha} d s|f(r)| d r+\int_0^t \int_0^r \int_0^q(t-s)^{-\alpha-1} d s \frac{|f(r)-f(q)|}{(r-q)^{1+\alpha}} d q d r\right) \cr
&\le&\Lambda\left(c_\alpha \int_0^t(t-r)^{-2 \alpha}|f(r)| d r+\int_0^t \int_0^r(t-q)^{-\alpha} \frac{|f(r)-f(q)|}{(r-q)^{1+\alpha}} d q d r\right),
\end{eqnarray*}
where $c_\alpha$ is same as in \eqref{alpha}.

Likewise, we define that 
\begin{eqnarray*}\label{3-7}
\mathcal{C}_3:=\int_0^t(t-s)^{-\alpha-1}\left|\int_s^t f(r) d u_r^\varepsilon\right| d s.
\end{eqnarray*}
Due to the fact that $\|u^\varepsilon\|_{1-\alpha,\infty} \le C \|u^\varepsilon\|_{1-\alpha+\kappa}\le C\|u^\varepsilon\|_{H}\le C \|u^\varepsilon\|_{\mathcal{H}^{H,d_1}}$, 
\begin{eqnarray*}\label{3-8}
\mathcal{C}_3 
&\le&\|u^\varepsilon\|_{\mathcal{H}^{H,d_1}}\left(c_\alpha \int_0^t(t-r)^{-2 \alpha}|f(r)| d r+\int_0^t \int_0^r(t-q)^{-\alpha} \frac{|f(r)-f(q)|}{(r-q)^{1+\alpha}} d q d r\right).
\end{eqnarray*}
Then,  we have
\begin{eqnarray}\label{3-9}
\|\tilde x^{(\varepsilon,\delta)}\|_{1, \lambda, t}&=& \sup _{0 \leq s \leq t} e^{-\lambda s} \int_0^s(s-r)^{-\alpha-1}\left|\int_r^s f_1( \tilde x_q^{(\varepsilon,\delta)}, \tilde y_q^{(\varepsilon,\delta)}) d q\right| d r \cr
&&+\sup _{0 \leq s \leq t} e^{-\lambda s} \int_0^s(s-r)^{-\alpha-1}\left|\int_r^s \sqrt{\varepsilon}\sigma_1( \tilde x_q^{(\varepsilon,\delta)}) d B_q^H\right| d r \cr
&&+\sup _{0 \leq s \leq t} e^{-\lambda s} \int_0^s(s-r)^{-\alpha-1}\left|\int_r^s \sigma_1( \tilde x_q^{(\varepsilon,\delta)}) d u_q^\varepsilon\right| d r \cr
&\le & C (\Lambda+\|u^\varepsilon\|_{\mathcal{H}^{H,d_1}})\left(1+\sup _{0 \leq s \leq t} \int_0^s e^{-\lambda(s-r)}\right.\cr
&&\left.\times\left[(s-r)^{-2 \alpha}\|\tilde x^{(\varepsilon,\delta)}\|_{\lambda, t}+(s-r)^{-\alpha}\|\tilde x^{(\varepsilon,\delta)}\|_{1, \lambda, t}\right] d r\right) \cr
&\le & C (\Lambda+\|u^\varepsilon\|_{\mathcal{H}^{H,d_1}})\left(1+\lambda^{2 \alpha-1}\|\tilde x^{(\varepsilon,\delta)}\|_{\lambda, t}+\lambda^{\alpha-1}\|\tilde x^{(\varepsilon,\delta)}\|_{1, \lambda, t}\right),
\end{eqnarray}
where the final inequality uses the estimate as follows,
\begin{eqnarray}\label{3-10}
\int_0^t e^{-\lambda(t-r)}(t-r)^{-2 \alpha} d r 
 \le \lambda^{2 \alpha-1} \int_0^{\infty} e^{-q} q^{-2 \alpha} d q 
 \le C \lambda^{2 \alpha-1} .
\end{eqnarray}
Take $\lambda=(4 C (\Lambda+\|u^\varepsilon\|_{\mathcal{H}^{H,d_1}}))^{\frac{1}{1-\alpha}}$, with the estimate (\ref{3-4}), it deduces that
\begin{eqnarray*}\label{3-11}
\|\tilde x^{(\varepsilon,\delta)}\|_{\lambda, t} \leq \frac{4}{3} C (\Lambda+\|u^\varepsilon\|_{\mathcal{H}^{H,d_1}})(1+\lambda^{-1}\|\tilde x^{(\varepsilon,\delta)}\|_{1, \lambda, t}),
\end{eqnarray*}
Substituting the above estimate into (\ref{3-9}), we can get that
\begin{eqnarray*}\label{3-12}
\|\tilde x^{(\varepsilon,\delta)}\|_{1, \lambda, t} &\le& \frac{3}{2} C (\Lambda+\|u^\varepsilon\|_{\mathcal{H}^{H,d_1}})+2(K (\Lambda+\|u^\varepsilon\|_{\mathcal{H}^{H,d_1}}))^{1 /(1-\alpha)} \cr
&\le& C (\Lambda+\|u^\varepsilon\|_{\mathcal{H}^{H,d_1}})^{1 /(1-\alpha)}.
\end{eqnarray*}
Moreover, we have
\begin{eqnarray*}\label{3-13}
\|\tilde x^{(\varepsilon,\delta)}\|_{ \lambda, t} 
&\le& C (\Lambda+\|u^\varepsilon\|_{\mathcal{H}^{H,d_1}})^{1 /(1-\alpha)}.
\end{eqnarray*}
Thus, it has
\begin{eqnarray*}\label{3-14}
\|\tilde x^{(\varepsilon,\delta)}\|_{\alpha, \infty} & \le& e^{\lambda T}\big(\|\tilde x^{(\varepsilon,\delta)}\|_{\lambda, T}+\|\tilde x^{(\varepsilon,\delta)}\|_{1, \lambda, T}\big) \cr
& \le& C \exp \big(C (\Lambda+\|u^\varepsilon\|_{\mathcal{H}^{H,d_1}})^{1 /(1-\alpha)}\big) (\Lambda+\|u^\varepsilon\|_{\mathcal{H}^{H,d_1}})^{1 /(1-\alpha)} ,
\end{eqnarray*}
according to the fact $0<\frac{1}{1-\alpha}<2$, assumptions $(u^{(\varepsilon,\delta)},v^{(\varepsilon,\delta)})\in \mathcal{A}_b^N$ and the  Fernique theorem for fBm, the conclusion is verified.
\qed
\begin{lem}\label{lem2}
	Assume (\textbf{A1})--(\textbf{A3}) and let $N\in(0,\infty)$. Then, there exists $C>0$ such that, for every $(u^{(\varepsilon,\delta)},v^{(\varepsilon,\delta)})\in \mathcal{A}_b^N$ and every $(t,h)$ with $0 \leq t \leq t+h \leq T$, we have 
	$$\mathbb{E}\big[|\tilde x_{t+h}^{(\varepsilon,\delta)}-\tilde x_{t}^{(\varepsilon,\delta)}|^2\big] \leq C h^{2-2 \alpha}.$$
	Here, $C$ is a positive constant which depends only on $N$.
\end{lem}
\para{Proof}. In this Proof, $C$ is a positive constant which depends only on  $N$ which may change from line to line. By taking some direct computations, we see that
\begin{eqnarray}\label{3-15}
\mathbb{E}\left[\left|\tilde x_{t+h}^{(\varepsilon,\delta)}-\tilde x_t^{(\varepsilon,\delta)}\right|^2\right] & \le& \mathbb{E}\bigg[\big|\int_t^{t+h} f_1( \tilde x_r^{(\varepsilon,\delta)}, \tilde y_r^{(\varepsilon,\delta)}) d r\big|^2\bigg]+\mathbb{E}\bigg[\big|\int_t^{t+h} \sigma_1( \tilde x_r^{(\varepsilon,\delta)}) d u_r^\varepsilon\big|^2\bigg] \cr
&& +\mathbb{E}\bigg[\varepsilon\big|\int_t^{t+h} \sigma_1(\tilde x_r^{(\varepsilon,\delta)}) d B_r^H\big|^2\bigg]\cr
&=:& A_1^h+A_2^h+A_3^h.
\end{eqnarray}
 {Due to  Assumption (\textbf{A2}) and \lemref{lem1}}, we get 
$A_1^h\le Ch^2.$
Then, it proceeds to show that
\begin{eqnarray}\label{3-17}
\big|\int_s^{t} \sigma_1( \tilde x_r^{(\varepsilon,\delta)}) d u_r^\varepsilon\big|&\le & \|u^\varepsilon\|_{\mathcal{H}^{H,d_1}}\left(\int_s^t \frac{|\sigma_1(\tilde x_r^{(\varepsilon,\delta)})|}{(r-s)^\alpha} d r+\int_s^t \int_s^r \frac{|\sigma_1( \tilde x_r^{(\varepsilon,\delta)})-\sigma_1( \tilde x_q^{(\varepsilon,\delta)})|}{(r-q)^{1+\alpha}} d q d r\right) \cr
&\le & \|u^\varepsilon\|_{\mathcal{H}^{H,d_1}}\left(1+\|\tilde x^{(\varepsilon,\delta)}\|_{\alpha, \infty}\right) \times\left(\int_s^t(r-s)^{-\alpha} d r \right) \cr
&\le & C \|u^\varepsilon\|_{\mathcal{H}^{H,d_1}}\big(1+\|\tilde x^{(\varepsilon,\delta)}\|_{\alpha, \infty}\big)(t-s)^{1-\alpha}.
\end{eqnarray}
So, we have 
\begin{eqnarray}\label{3-18}
A_2^h \leq C \mathbb{E}\big[\|u^\varepsilon\|_{\mathcal{H}^{H,d_1}}^2(1+\|\tilde x^{(\varepsilon,\delta)}\|_{\alpha, \infty})^2\big] h^{2-2 \alpha}.
\end{eqnarray}
Likewise, we can get that
\begin{eqnarray}\label{3-19}
\big|\int_s^{t} \sigma_1( \tilde x_r^{(\varepsilon,\delta)}) d B_r^H\big|&\le & \Lambda\left(\int_s^t \frac{|\sigma_1(\tilde x_r^{(\varepsilon,\delta)})|}{(r-s)^\alpha} d r+\int_s^t \int_s^r \frac{|\sigma_1( \tilde x_r^{(\varepsilon,\delta)})-\sigma_1( \tilde x_q^{(\varepsilon,\delta)})|}{(r-q)^{1+\alpha}} d q d r\right) \cr
&\le & C \Lambda\big(1+\|\tilde x^{(\varepsilon,\delta)}\|_{\alpha, \infty}\big)(t-s)^{1-\alpha},
\end{eqnarray}
therefore, 
\begin{eqnarray}\label{3-20}
A_3^h \leq  C \mathbb{E}\big[\Lambda^2(1+\|\tilde x^{(\varepsilon,\delta)}\|_{\alpha, \infty})^2\big] h^{2-2 \alpha}.
\end{eqnarray}
If we combine (\ref{3-15})--(\ref{3-20}) and \lemref{lem1}, then we can obtain the desired estimate.\qed

Before giving estimates of  $\tilde  y^{(\varepsilon,\delta)}$, we need to show the following estimate.
\begin{lem}\label{lem4}
	Assume (\textbf{A1})--(\textbf{A3}). Let $N\in(0,\infty)$. Then,
	 for every $(u^{(\varepsilon,\delta)},v^{(\varepsilon,\delta)})\in \mathcal{A}_b^N$, 
	$\sup_{0\le s\le t}|\tilde  y_s^{(\varepsilon,\delta)}|$ has moments of all orders.
\end{lem}
\para{Proof}. In this proof,   $0<\delta<\varepsilon\le 1$ are fixed and $C$ is a positive constant which depends  on  $N$, $p$, $\varepsilon$ and $\delta$ which may change from line to line. We define $(X_t,Y_t):=(\tilde  x_t^{(\varepsilon,\delta)},\tilde  y_t^{(\varepsilon,\delta)})$, $X_t^{*}:=\sup_{0\le s\le t}|\tilde  x_s^{(\varepsilon,\delta)}|$ and $Y_t^{*}:=\sup_{0\le s\le t}|\tilde  y_s^{(\varepsilon,\delta)}|$.   

For every $1\le p<\infty$, we have
\begin{eqnarray}\label{4-1}
\begin{aligned}
\mathbb{E}[|Y_t^{*}|^p] &\le C {| y_0 |^p} + C\mathbb{E}\big[|\sup_{0 \leq s \leq t}|\int_0^s f_2( X_r,Y_r )dr||^p \big]
+  C\mathbb{E}\big[|\sup_{0 \leq s \leq t}|\int_0^s  \sigma_{2} ( X_r,Y_r ){\frac{dv_r^{(\varepsilon, \delta)}}{dr}  } dr ||^p\big] \\
&\quad+ C\mathbb{E}\big[\big|\sup_{0 \leq s \leq t}\big|\int_0^s \sigma_{2}( X_r,Y_r) dW_r\big|\big|^p\big] \\
&:=C {| y_0 |^p} + K_1+K_2+K_3.
\end{aligned}
\end{eqnarray}
By Assumption (\textbf{A2}), we have 
\begin{eqnarray}\label{4-2}
\begin{aligned}
K_1 &\le  C+C[\sup_{0 \leq s \leq t}|X_s|]^p+C\int_{0}^{t}[\sup_{0 \leq r \leq s}|Y_r|]^pds .
\end{aligned}
\end{eqnarray}
Likewise, by Assumption (\textbf{A2}) and H\"older inequality, we obtain
\begin{eqnarray}\label{4-3}
\begin{aligned}
K_2 &\le  C\mathbb{E}\big[[\sup_{0 \leq s \leq t}\int_0^s  |\sigma_{2} ( X_r,Y_r)|^2dr]^{p/2} \times[\int_0^s|{\frac{dv_r^{(\varepsilon, \delta)}}{dr}  }|^2 dr ]^{p/2}\big] \\
&\le C\big\{\mathbb{E}\big[\sup_{0 \leq s \leq t}\int_0^s  |\sigma_{2} ( X_r,Y_r)|^2dr\big]^{p}\big\}^{\frac{1}{2}} \times C\big\{\mathbb{E}\big[\sup_{0 \leq s \leq t}\int_0^s|{\frac{dv_r^{(\varepsilon, \delta)}}{dr}  }|^2 dr \big]^{p}\big\}^{\frac{1}{2}} \\
&\le C\big\{\mathbb{E}[|1+[\sup_{0 \leq s \leq t}|X_s|]^p|]\big\}^{\frac{1}{2}}   \\
&\le C\big\{\mathbb{E}[1+[X_t^{*}]^p]\big\}^{\frac{1}{2}}.
\end{aligned}
\end{eqnarray}
For the final term, with aid of  Burkholder-Davis-Gundy inequality, we have
\begin{eqnarray}\label{4-4}
\begin{aligned}
K_3 &\le  C\mathbb{E}\big[\big|\sup_{0 \leq s \leq t}\big|\int_0^s\sigma_{2}( X_r,Y_r) dW_r\big|\big|^p\big] \\
&\le C\mathbb{E}\big[\big|\int_0^t |\sigma_{2}( X_r,Y_r)|^2 dr\big|^{p/2}\big] \\
&\le C\mathbb{E}[|1+[\sup_{0 \leq s \leq t}|X_s|]^p|]  \\
&\le C+C\mathbb{E}[[X_t^{*}]^p].
\end{aligned}
\end{eqnarray}
By combining (\ref{4-1}) to (\ref{4-4}) and \lemref{lem1}, we have
\begin{eqnarray}\label{4-5}
\begin{aligned}
\mathbb{E}[|Y_t^{*}|^p] &\le  C\int_{0}^{t}\mathbb{E}[|Y_s^{*}|^p]ds+
C,\quad t\in[0,T].
\end{aligned}
\end{eqnarray}
Then, by the Gronwall inequality, we have
\begin{eqnarray*}\label{4-6}
\begin{aligned}
		\mathbb{E}[|Y_t^{*}|^p] &\le C e^{C	t},\quad t\in[0,T].
\end{aligned}
\end{eqnarray*}
The proof is completed. \qed
\begin{lem}\label{lem3}
	Assume (\textbf{A1})--(\textbf{A3}). We additionally assume $\frac{\delta}{\varepsilon}\le (\frac{\beta_2}{2})^2$. Let $N\in(0,\infty)$. Then, there exists $C>0$ such that for every $(u^{(\varepsilon,\delta)},v^{(\varepsilon,\delta)})\in \mathcal{A}_b^N$, we have
	$$\int_{0}^{T} \mathbb{E}\big[|\tilde  y_t^{(\varepsilon,\delta)}|^2\big] dt\leq C.$$
	Here, $C$ is a positive constant which depends only  on $N$.
\end{lem}
\para{Proof}. By the It\^o's formula, we have
\begin{eqnarray*}
\begin{aligned}
\mathbb{E}[{| \tilde {y}_{t}^{(\varepsilon,\delta)} |^2}] &= {| y_0 |^2} + \frac{2}{\delta }\mathbb{E}\int_0^t {\langle \tilde {y}_{s}^{(\varepsilon,\delta)},f_2( \tilde {x}_{s}^{(\varepsilon,\delta)},\tilde {y}_{s}^{(\varepsilon,\delta)} )\rangle ds} 
+ \frac{2}{{\sqrt {\delta \varepsilon } }}\mathbb{E}\int_0^t {\langle \tilde {y}_{s}^{(\varepsilon,\delta)},\sigma_{2} ( {\tilde {x}_{s}^{(\varepsilon,\delta)},\tilde {y}_{s}^{(\varepsilon,\delta)}} ){\frac{dv_s^{(\varepsilon, \delta)}}{ds}  }\rangle ds}   \\
&\quad+\frac{2}{\sqrt{\delta} }\mathbb{E}\int_0^t {\langle \tilde {y}_{s}^{(\varepsilon,\delta)},\sigma_{2}( {\tilde {x}_{s}^{(\varepsilon,\delta)},\tilde {y}_{s}^{(\varepsilon,\delta)}} ) dW_s\rangle}+ \frac{1}{\delta }\mathbb{E}\int_0^t |{\sigma_{2} }( {\tilde {x}_{s}^{(\varepsilon,\delta)},\tilde {y}_{s}^{(\varepsilon,\delta)}} )|^2ds,
\end{aligned}
\end{eqnarray*}
where $|\sigma_2 (x,y)|$ is the Hilbert-Schmidt norm of the matrix $\sigma_2 (x,y)$.

According to \lemref{lem1} and \lemref{lem4}, we can obtain that the fourth term is a true martingale. In particular, we have $\mathbb{E}[\int_0^t {\langle \tilde {y}_{s}^{(\varepsilon,\delta)},\sigma_{2}( {\tilde {x}_{s}^{(\varepsilon,\delta)},\tilde {y}_{s}^{(\varepsilon,\delta)}} ) dW_s\rangle}]=0$.
So, we have
\begin{eqnarray}\label{3-21}
\begin{aligned}
\frac{d\mathbb{E}[{| \tilde {y}_{t}^{(\varepsilon,\delta)} |^2}]}{dt} &= \frac{2}{\delta }\mathbb{E} {\langle \tilde {y}_{t}^{(\varepsilon,\delta)},f_2( \tilde {x}_{t}^{(\varepsilon,\delta)},\tilde {y}_{t}^{(\varepsilon,\delta)} )\rangle } 
+ \frac{2}{{\sqrt {\delta \varepsilon } }}\mathbb{E} {\langle \tilde {y}_{t}^{(\varepsilon,\delta)},\sigma_{2} ( {\tilde {x}_{t}^{(\varepsilon,\delta)},\tilde {y}_{t}^{(\varepsilon,\delta)}} ){\frac{dv_t^{(\varepsilon, \delta)}}{dt}  }\rangle }   \\
&\quad+ \frac{1}{\delta }\mathbb{E} |{\sigma_{2} }( {\tilde {x}_{t}^{(\varepsilon,\delta)},\tilde {y}_{t}^{(\varepsilon,\delta)}} )|^2,
\end{aligned}
\end{eqnarray} 
With Assumption (\textbf{A3}), we have
\begin{eqnarray}\label{lemma3.12}
\frac{2}{\delta }{\langle \tilde {y}_{t}^{(\varepsilon,\delta)},f_2( \tilde {x}_{t}^{(\varepsilon,\delta)},\tilde {y}_{t}^{(\varepsilon,\delta)} )\rangle } +\frac{1}{\delta } |{\sigma_{2} }( {\tilde {x}_{t}^{(\varepsilon,\delta)},\tilde {y}_{t}^{(\varepsilon,\delta)}} )|^2\le- \frac{{ \beta_2 }}{\delta }{{| \tilde {y}_{t}^{(\varepsilon,\delta)} |^2}}  + {\frac{{ C }}{\delta }{{| \tilde {x}_{t}^{(\varepsilon,\delta)} |^2}} }+ \frac{{C}}{\delta }.	
\end{eqnarray}
By  Assumption (\textbf{A2}) and  $(u^{(\varepsilon, \delta)}, v^{(\varepsilon, \delta)})\in \mathcal{A}_{b}^N$, then, we have
\begin{eqnarray}\label{3-22}
\begin{aligned}
\frac{2}{{\sqrt {\delta \varepsilon } }}{\langle \tilde {y}_{t}^{(\varepsilon,\delta)},\sigma_{2} ( {\tilde {x}_{t}^{(\varepsilon,\delta)},\tilde {y}_{t}^{(\varepsilon,\delta)}} ){\frac{dv_t^{(\varepsilon, \delta)}}{dt}  }\rangle }\le   \frac{{{L}}}{{\sqrt {\delta \varepsilon } }}\big( 1 + | \tilde {x}_{t}^{(\varepsilon,\delta)} |^2 \big)| {\frac{dv_t^{(\varepsilon, \delta)}}{dt}  } |^2+   \frac{{{1 }}}{{\sqrt {\delta \varepsilon } }} | \tilde {y}_{t}^{(\varepsilon,\delta)} |^2 .
\end{aligned}
\end{eqnarray}
Thus, from (\ref{3-21}) and (\ref{3-22}), we have
\begin{eqnarray*}\label{3-23}
\begin{aligned}
\frac{d\mathbb{E}[{| \tilde {y}_{t}^{(\varepsilon,\delta)} |^2}]}{dt} &\le  ({\frac{{ - \beta_2 }}{\delta } }+\frac{{{1 }}}{{\sqrt {\delta \varepsilon } }}) \mathbb{E}[{| \tilde {y}_{t}^{(\varepsilon,\delta)} |^2}]  +   \frac{{{L }}}{{\sqrt {\delta \varepsilon } }}  \mathbb{E}[{{| \tilde {x}_{t}^{(\varepsilon,\delta)} |^2{| {\frac{dv_t^{(\varepsilon, \delta)}}{dt}  } |^2}}}]  +   \frac{{{L }}}{\sqrt {\delta \varepsilon } } \mathbb{E}[{| {\frac{dv_t^{(\varepsilon, \delta)}}{dt}  } |^2}] + {\frac{{ C }}{\delta }{{\mathbb{E}[| \tilde {x}_{t}^{(\varepsilon,\delta)} |^2]}} }+ \frac{{ C}}{\delta }\\
&\le  {\frac{{ - \beta_2 }}{2\delta } } \mathbb{E}[{| \tilde {y}_{t}^{(\varepsilon,\delta)} |^2}]  +   \frac{{{L }}}{{\sqrt {\delta \varepsilon } }}  \mathbb{E}[{{| \tilde {x}_{t}^{(\varepsilon,\delta)} |^2{| {\frac{dv_t^{(\varepsilon, \delta)}}{dt}  } |^2}}}]  +   \frac{{{L }}}{\sqrt {\delta \varepsilon } } \mathbb{E}[{| {\frac{dv_t^{(\varepsilon, \delta)}}{dt}  } |^2}] + {\frac{{ C }}{\delta }{{\mathbb{E}[| \tilde {x}_{t}^{(\varepsilon,\delta)} |^2]}} }+ \frac{{ C}}{\delta }.
\end{aligned}
\end{eqnarray*}
Moreover, consider the ODE  as following:
\begin{eqnarray*}\label{3-233}
\begin{aligned}
\frac{d {A}_{t} }{dt} 
&=  {\frac{{ - \beta_2 }}{2\delta } } {{A}_{t}  } +   \frac{{{L }}}{{\sqrt {\delta \varepsilon } }}  \mathbb{E}[{{| \tilde {x}_{t}^{(\varepsilon,\delta)} |^2{| {\frac{dv_t^{(\varepsilon, \delta)}}{dt}  } |^2}}}]  +   \frac{{{L }}}{\sqrt {\delta \varepsilon } } \mathbb{E}[{| {\frac{dv_t^{(\varepsilon, \delta)}}{dt}  } |^2}] + {\frac{{ C }}{\delta }{{\mathbb{E}[| \tilde {x}_{t}^{(\varepsilon,\delta)} |^2]}} }+ \frac{{ C}}{\delta }
\end{aligned}
\end{eqnarray*}
with initial value $ A_0=|y_0|^2$. So the solution has an explicit  expression  as follows:
\begin{eqnarray*}\label{3-234}
\begin{aligned}
{A_t}
&=  |y_0|^2 e^{-\frac{\beta_2}{2\delta} t} +   \frac{{{L }}}{{\sqrt {\delta \varepsilon } }}\int_{0}^{t}  e^{-\frac{\beta_2}{2\delta} (t-s)}{\mathbb{E}[{| \tilde {x}_{s}^{(\varepsilon,\delta)} |^2{| {\frac{dv_s^{(\varepsilon, \delta)}}{ds}  } |^2}}]} ds+   \frac{{{L }}}{\sqrt {\delta \varepsilon } }\int_{0}^{t}  e^{-\frac{\beta_2}{2\delta} (t-s)} { \mathbb{E}[|{\frac{dv_s^{(\varepsilon, \delta)}}{ds}  } |^2]}ds\\
&\quad+ {\frac{{ C }}{\delta }\int_{0}^{t}  e^{-\frac{\beta_2}{2\delta} (t-s)}{{\mathbb{E}[| \tilde {x}_{s}^{(\varepsilon,\delta)} |^2]}} ds} + \frac{{ C}}{\delta }\int_{0}^{t}  e^{-\frac{\beta_2}{2\delta} (t-s)}ds.
\end{aligned}
\end{eqnarray*}
Besides, by some simple computation, we can obtain that 
\begin{eqnarray*}\label{3-235}
\begin{aligned}
\frac{d(\mathbb{E}[{| \tilde {y}_{t}^{(\varepsilon,\delta)} |^2}]-{A_t})}{dt} 
&\le{\frac{{ - \beta_2 }}{2\delta } } \big[\mathbb{E}[{| \tilde {y}_{t}^{(\varepsilon,\delta)} |^2}]-{A_t}]\big]. 
\end{aligned}
\end{eqnarray*}
So, by comparison theorem, we have for all $t$ that
\begin{eqnarray*}\label{3-231}
\begin{aligned}
\mathbb{E}[{| \tilde {y}_{t}^{(\varepsilon,\delta)} |^2}] \le A_t
&= |y_0|^2 e^{-\frac{\beta_2}{2\delta} t} +   \frac{{{L }}}{{\sqrt {\delta \varepsilon } }}\int_{0}^{t}  e^{-\frac{\beta_2}{2\delta} (t-s)}{\mathbb{E}[{| \tilde {x}_{s}^{(\varepsilon,\delta)} |^2{| {\frac{dv_s^{(\varepsilon, \delta)}}{ds}  } |^2}}]} ds+   \frac{{{L }}}{\sqrt {\delta \varepsilon } }\int_{0}^{t}  e^{-\frac{\beta_2}{2\delta} (t-s)} { \mathbb{E}[|{\frac{dv_s^{(\varepsilon, \delta)}}{ds}  } |^2]}ds  \\
&\qquad + {\frac{{ C }}{\delta }\int_{0}^{t}  e^{-\frac{\beta_2}{2\delta} (t-s)}{{\mathbb{E}[| \tilde {x}_{s}^{(\varepsilon,\delta)} |^2]}} ds}+ \frac{{ C}}{\delta }\int_{0}^{t}  e^{-\frac{\beta_2}{2\delta} (t-s)}ds.
\end{aligned}
\end{eqnarray*}
Then, with Fubini theorem, we can obtain
\begin{eqnarray*}\label{3-232}
\begin{aligned}
\int_{0}^{T}\mathbb{E}[{| \tilde {y}_{t}^{(\varepsilon,\delta)} |^2}]dt 
&\le |y_0|^2 \int_{0}^{T} e^{-\frac{\beta_2}{2\delta} t} dt+   \frac{{{L }}}{{\sqrt {\delta \varepsilon } }}\int_{0}^{T}\int_{0}^{t}  e^{-\frac{\beta_2}{2\delta} (t-s)}{\mathbb{E}[{| \tilde {x}_{s}^{(\varepsilon,\delta)} |^2{| {\frac{dv_s^{(\varepsilon, \delta)}}{ds}  } |^2}}]} dsdt  \\
&\quad+   \frac{{{L }}}{\sqrt {\delta \varepsilon } }\int_{0}^{T}\int_{0}^{t}  e^{-\frac{\beta_2}{2\delta} (t-s)} { \mathbb{E}[|{\frac{dv_s^{(\varepsilon, \delta)}}{ds}  } |^2]}ds + {\frac{{ C }}{\delta }\int_{0}^{T}\int_{0}^{t}  e^{-\frac{\beta_2}{2\delta} (t-s)}{{\mathbb{E}[| \tilde {x}_{s}^{(\varepsilon,\delta)} |^2]}} dsdt}\\
&\quad+\frac{{ C}}{\delta }\int_{0}^{T}\int_{0}^{t}  e^{-\frac{\beta_2}{2\delta} (t-s)}dsdt\\
&\le |y_0|^2 e^{-\frac{\beta_2}{2\delta} T} +   \frac{{{L  }}}{{\sqrt {\delta \varepsilon } }}\mathbb{E}\big[\sup_{0 \leq t \leq T}| \tilde {x}_{t}^{(\varepsilon,\delta)} |^2\int_{0}^{T}\int_{s}^{T}  e^{-\frac{\beta_2}{2\delta} (t-s)} dt| {\frac{dv_s^{(\varepsilon, \delta)}}{ds}  } |^2ds\big]  \\
&\quad+   \frac{{{L }}}{\sqrt {\delta \varepsilon } }\int_{0}^{T}\int_{s}^{T}  e^{-\frac{\beta_2}{2\delta} (t-s)}dt { \mathbb{E}[|{\frac{dv_s^{(\varepsilon, \delta)}}{ds}  } |^2]}ds + {\frac{{ C }}{\delta }\int_{0}^{T}\int_{0}^{t}  e^{-\frac{\beta_2}{2\delta} (t-s)}{{\mathbb{E}[| \tilde {x}_{s}^{(\varepsilon,\delta)} |^2]}} dsdt}\\
&\quad+ \frac{{ C}}{\delta }\int_{0}^{T}\int_{0}^{t}  e^{-\frac{\beta_2}{2\delta} (t-s)}ds\\
&\le |y_0|^2 e^{-\frac{\beta_2}{2\delta} T} +   \frac{2\delta L}{\beta_2\sqrt {\delta \varepsilon } }\mathbb{E}\big[\sup_{0 \leq t \leq T}| \tilde {x}_{t}^{(\varepsilon,\delta)} |^2\int_{0}^{T}  e^{-\frac{\beta_2}{2\delta} (T-s)} | {\frac{dv_s^{(\varepsilon, \delta)}}{ds}  } |^2ds\big]  \\
&\quad+   \frac{2\delta L}{\beta_2\sqrt {\delta \varepsilon } }\int_{0}^{T} e^{-\frac{\beta_2}{2\delta} (T-s)} { \mathbb{E}[|{\frac{dv_s^{(\varepsilon, \delta)}}{ds}  } |^2]}ds + {C{\mathbb{E}[\sup_{ t\in[0,T]}| \tilde {x}_{t}^{(\varepsilon,\delta)} |^2]}\int_{0}^{T} e^{-\frac{\beta_2}{2\delta} (T-s)} ds}+ C.
\end{aligned}
\end{eqnarray*}
By the fact that $(u^{(\varepsilon, \delta)}, v^{(\varepsilon, \delta)})\in \mathcal{A}_{b}^N$, we have
\begin{eqnarray*}\label{3-24}
\int_{0}^{T}\mathbb{E}[{| \tilde {y}_{t}^{(\varepsilon,\delta)} |^2}]dt &\le& {C}\mathbb{E}[\mathop {\sup }\limits_{t \in \left[ {0,T} \right]} {| \tilde {x}_{t}^{(\varepsilon,\delta)} |^2}]+C.
\end{eqnarray*}
Combining this with \lemref{lem1}, we finish the proof. 
\qed

\section{Proof of Main Theorem (\thmref{thm})}\label{sec.5}
	In this Section, $H$ and $\alpha$ are such that $H\in(1/2,1)$ and  $1-H<\alpha<1/2$. We divide this section into three steps.

	\textbf{Step 1}.  Everything in this step is deterministic.
	Let $(u^{(j)}, v^{(j)})$ and $(u, v)$ belong to $S_N$ such that $(u^{(j)}, v^{(j)})\rightarrow(u, v)$ as $j\rightarrow\infty$ in the weak topology of $\mathcal{H}$. 
	In this step, it proceeds to show the following convergence:
	\begin{eqnarray} \label{3-26}
	\mathcal{G}^{0}( u^{(j)} ,v^{(j)} )\rightarrow\mathcal{G}^{0}(u,v) \label{step2}
	\end{eqnarray}
in $C^{\alpha}([0,T],\mathbb{R}^m)$ as $j \to \infty$.

	If $(u^{(j)}, v^{(j)})\in\mathcal{H}$, 	by the bounded embedding that 
	$\mathcal{H}^{H,d_1}\subset C^{1-\alpha}([0,T], \mathbb{R}^{d_1})$,  
	 we have the fact that $\{u^{(j)}\}_{j\ge 1} \in C^{1-\alpha}([0,T], \mathbb{R}^{d_1})$. Let $\{\tilde {x}^{(j)}_t\}_{j\ge 1}$ be a family of solutions to the equation (\ref{3}), that is,
	\begin{eqnarray*} \label{3-25}
	d\tilde {x}^{(j)}_t =  \bar{f}_1( \tilde {x}^{(j)}_t)dt + \sigma_{1}(\tilde {x}^{(j)}_t)du^{(j)}_t,\quad \tilde {x}^{(j)}_0=x_0.
	\end{eqnarray*}	
	There exists a unique solution $\{\tilde {x}^{(j)}\} $  to the skeleton equation (\ref{3}), and it is bounded in   $ C^{1-\alpha}([0,T], \mathbb{R}^{m})$ \cite[Proposition 3.6]{2020Budhiraja}. In other words, there is a constant $C$ such that
	\begin{equation*} \label{3-28}
	\|\tilde x^{(j)}\|_{1-\alpha} \leq C. 
	\end{equation*}	

	Since $\alpha<1-\alpha$, we have a compact embedding  $C^{1-\alpha}([0,T],\mathbb{R}^m) \subset  C^{\alpha}([0,T],\mathbb{R}^m)$. Therefore, the family $\{\tilde x^{(j)}\}_{ j\ge 1}$ is pre-compact in $ C^{\alpha}([0,T],\mathbb{R}^m)$. Let $\tilde x$ be any limit point. Then, there is  a subsequence of $\{\tilde x^{(j)}\}_{ j\ge 1}$ weakly converges to $\tilde x$ in $ C^{\alpha}([0,T],\mathbb{R}^m)$, which will be denoted by the same symbol. 
We will show that the limit point $\tilde x$ satisfies the  differential equation as follows,
	\begin{eqnarray} \label{3-44}
	d\tilde {x}_t =  \bar{f}_1( \tilde {x}_t)dt + \sigma_{1}(\tilde {x}_t)du_t \quad \tilde {x}_0=x_0.
	\end{eqnarray}
The uniqueness of (\ref{3-44}) implies the uniqueness of the limit point $\tilde x$.
	
	For any $u=\mathcal{K_H} {\dot u} \in \mathcal{H}^{H,d_1}$, $u$ is differentiable and 
	\begin{eqnarray}\label{3-29}
	u^{\prime}(t)=c_H t^{H-\frac{1}{2}}\left(I_{0+}^{H-\frac{1}{2}}\left(\psi^{-1} \dot{u}\right)\right)(t)=\frac{c_H t^{H-\frac{1}{2}}}{\Gamma\left(H-\frac{1}{2}\right)} \int_0^t(t-s)^{H-\frac{3}{2}} s^{\frac{1}{2}-H} \dot{u}(s) d s,
	\end{eqnarray}
	where $\psi(u)=u^{H-\frac{1}{2}}$.
	
	By the direct computation, we get that
	\begin{eqnarray*}\label{3-30}
|\tilde x^{(j)}_t-\tilde x_t| &\le&|\int_0^t (\bar f_1(\tilde x_s^{(j)}) -\bar f_1(\tilde x_s))d s | +|\int_0^t (\sigma_1( \tilde x_s^{(j)})-\sigma_1( \tilde x_s)) d u_s^{(j)} | \cr
&&+|\int_0^t \sigma_1( \tilde x_s) (d u_s^{(j)}-d u_s) |\cr
	&=:& K_1+K_2+K_3.
	\end{eqnarray*}
For the term $K_1$, by the definition of $\alpha$-H\"older norm and Assumption {(\bf{A2})}, it deduces that 
\begin{eqnarray*}\label{3-31}
K_1 ={\big|\int_0^t (\bar f_1(\tilde x_s^{(j)}) -\bar f_1(\tilde x_s))d s \big|} \le C\int_0^t |\tilde x_s^{(j)} -\tilde x_s|d s \le C\sup_{ t\in[0,T]}|\tilde x_t^{(j)} -\tilde x_t| .
\end{eqnarray*}
Then, it proceeds to estimate the second term $K_2$. By the differential formula (\ref{3-29}), Assumption {(\bf{A1})}, and Fubini theorem,  we have
\begin{eqnarray*}\label{3-32}
K_2 &=&{\big|\int_0^t ( \sigma_1(\tilde x_s^{(j)}) - \sigma_1(\tilde x_s))d u^{(j)}_s \big|} \cr
&\le&{\bigg|\int_0^t(\sigma(x_s^{(j)})-\sigma(x_s^{(j)})) s^{\frac{1}{2}-H}\big[\int_0^s(s-r)^{H-\frac{3}{2}} r^{H-\frac{1}{2}} \dot u^{(j)}_rdr\big]  ds \bigg|} \cr
&\le&C\sup_{ t\in[0,T]}|\tilde x_t^{(j)} -\tilde x_t|\times{\bigg|\int_0^t r^{\frac{1}{2}-H}\big[\int_r^t(s-r)^{H-\frac{3}{2}} s^{H-\frac{1}{2}}ds\big]  |\dot u^{(j)}_r|dr \bigg|} \cr
&\le&C\sup_{ t\in[0,T]}|\tilde x_t^{(j)} -\tilde x_t|\times{\bigg|\int_0^t r^{\frac{1}{2}-H} |\dot u^{(j)}_r|dr \bigg|} \cr
&\le&C\sup_{ t\in[0,T]}|\tilde x_t^{(j)} -\tilde x_t|\|u^{(j)}\|_{\mathcal{H}^{H,d_1}}\cr
&\le& C\sqrt{2N}\sup_{ t\in[0,T]}|\tilde x_t^{(j)} -\tilde x_t|.
\end{eqnarray*}
Since  $\{\tilde x^{(j)}\}_{ j\ge 1}$ converges to $\tilde x$ in the uniform norm, we have
$K_1+ K_1\to0$ as $j \to \infty$.

Finally, we compute the third term $K_3$. By Fubini theorem, we have
\begin{eqnarray*}\label{3-35}
K_3 &=&{\big|\int_0^t \sigma_1(\tilde x_s)(d u^{(j)}_s-d u_s) \big|} \cr
&=&{\bigg|\int_0^t \sigma(\tilde x_s) s^{H-\frac{1}{2}}\big[\int_0^s(s-r)^{H-\frac{3}{2}} r^{\frac{1}{2}-H}(\dot u^{(j)}_r-\dot u_r)dr\big] ds \bigg|} \cr
&=&{\bigg|\int_0^T \mathbf{1}_{[0,t]}(r) r^{\frac{1}{2}-H}\big[\int_r^t(s-r)^{H-\frac{3}{2}} s^{H-\frac{1}{2}}\sigma(\tilde x_s)ds\big]  (\dot u^{(j)}_r-\dot u_r)dr \bigg|} .
\end{eqnarray*}
Set
$ K_r(t)=\big[\int_r^t(s-r)^{H-\frac{3}{2}} s^{H-\frac{1}{2}}\sigma(\tilde x_s)ds\big] $, then we have
$|K_r(t)|\le C(1+\|\tilde x\|)$.
Since $(\dot u^{(j)}, \dot v^{(j)})\rightarrow(\dot u, \dot v)$ as $j\rightarrow\infty$ in the weak topology of $L^2$, then $K_3$ converges to 0 as $j\rightarrow\infty$.

Therefore, we can get that the limit point $\tilde x$ satisfies the differential equation (\ref{3}).Therefore, we get that  $\{\tilde x^{(j)}\}_{ j\ge 1}$ weakly converges to $\tilde x$ in $ C^{\alpha}([0,T],\mathbb{R}^m)$.

	\textbf{Step 2}. In this step, we make probabilistic arguments. 
	Let $0<N<\infty$ and 
	 assume $0<\delta<\varepsilon$, $\delta=o(\varepsilon)$ and we will let $\varepsilon \to 0$.
	
	 Assume  $(u^{(\varepsilon, \delta)}, v^{(\varepsilon, \delta)})\in \mathcal{A}^{N}_b$ such that $(u^{(\varepsilon, \delta)}, v^{(\varepsilon, \delta)})$ weakly converges  to $(u, v)$ as $\varepsilon \to 0$. Reformulating the slow variables of system (\ref{2}) as follows,
	\[
	\tilde {x}^{(\varepsilon,\delta)}:=\mathcal{G}^{(\varepsilon,\delta)}(\sqrt \varepsilon B^H+u^{(\varepsilon, \delta)}, \sqrt \varepsilon W+v^{(\varepsilon, \delta)}),
	\]
	we will show that  as $\varepsilon\rightarrow 0$, $\tilde {x}^{(\varepsilon,\delta)}$ weakly converges to $\tilde {x}$ in $ C^{\alpha}([0,T],\mathbb{R}^m)$, that is,
	\begin{eqnarray} \label{step3}
	\mathcal{G}^{(\varepsilon,\delta)}(\sqrt \varepsilon B^H+u^{(\varepsilon, \delta)}, \sqrt \varepsilon W+v^{(\varepsilon, \delta)})\xrightarrow{\textrm{weakly}}\mathcal{G}^{0}(u , v).
	\end{eqnarray}	
	 To show (\ref{step3}), construct the auxiliary processes as follows:
	\begin{eqnarray*}
	\left
	\{
	\begin{array}{ll}
	d\hat {x}^{(\varepsilon, \delta)}_t =& f_{1}(\tilde {x}^{(\varepsilon, \delta)}_{t(\Delta)}, \hat {y}^{(\varepsilon, \delta)}_t)dt + \sigma_{1}(\hat {x}^{(\varepsilon, \delta)}_t)du^{(\varepsilon, \delta)}_t,\\
	\delta d\hat {y}^{(\varepsilon, \delta)}_t =&  f_{2}( \tilde {x}^{(\varepsilon, \delta)}_{t(\Delta)}, \hat {y}^{(\varepsilon, \delta)}_t)dt + {\sqrt \delta }\sigma_2(\tilde {x}^{(\varepsilon, \delta)}_{t(\Delta)}, \hat {y}^{(\varepsilon, \delta)}_t)dW_{t}
	\end{array}
	\right.
	\end{eqnarray*}
		with the same intial condition as (\ref{2}), where $t(\Delta)=\left\lfloor\frac{t}{\Delta}\right\rfloor \Delta$ is the nearest breakpoint preceding $t$. Note that in $\hat x_t^{(\varepsilon,\delta)}$ is substituted in $\sigma_1$. 
				
By taking same manner in \lemref{lem1} and \lemref{lem2}, we have
\begin{eqnarray}\label{3-36}
\mathbb{E}\big[\|\hat x^{(\varepsilon,\delta)}\|_{\alpha, \infty}^p\big] \leq C
\end{eqnarray}
and 
\begin{eqnarray*}\label{3-37}
\int_{0}^{T} \mathbb{E}\big[|\hat y_t^{(\varepsilon,\delta)}|^2\big] dt\leq C,
\end{eqnarray*}	
where the constant $C$ is independent of $\varepsilon, \delta,\Delta$.

From now, we will estimate 
\begin{eqnarray*}\label{3-38}
\mathbb{E}\big[\sup _{t \in[0, T]}\|\tilde x^{(\varepsilon,\delta)}-\hat{x}^{(\varepsilon,\delta)}\|_{\alpha,[0,t]}^2\big].
\end{eqnarray*}
Let $R>0$ be large enough and define the stopping time $\tau_R:=\inf \left\{t \in[0,T]:\|B^H\|_{1-\alpha, \infty, t} \geq R\right\} \wedge T$.
Firstly, we can get that
\begin{eqnarray}\label{b3-40}
\mathbb{E}\Big[\sup _{t \in[0, T]}\left\|\tilde {x}^{(\varepsilon,\delta)}-\hat{x}^{(\varepsilon,\delta)}\right\|_{\alpha,[0,t]}^2\Big] \leq & \mathbb{E}\left[\sup _{t \in[0, T]}\left\|\tilde {x}^{(\varepsilon,\delta)}-\hat{x}^{(\varepsilon,\delta)}\right\|_{\alpha,[0,t]}^2 I_{\{\tau_R \le T\}}\right] \cr
&+\mathbb{E}\left[\sup _{t \in[0, T]}\left\|\tilde {x}^{(\varepsilon,\delta)}-\hat{x}^{(\varepsilon,\delta)}\right\|_{\alpha,[0,t]}^2I_{\{\tau_R > T\}}\right].
\end{eqnarray}
For the first term, by the H\"older inequality, we have
\begin{eqnarray}\label{b3-41}
\mathbb{E}\Big[\sup _{t \in[0, T]}\left\|\tilde {x}^{(\varepsilon,\delta)}-\hat{x}^{(\varepsilon,\delta)}\right\|_{\alpha,[0,t]}^2 I_{\{\tau_R \le T\}}\Big] &\leq & \mathbb{E}\Big[\sup _{t \in[0, T]}\left\|\tilde {x}^{(\varepsilon,\delta)}-\hat{x}^{(\varepsilon,\delta)}\right\|_{\alpha,[0,t]}^4 \Big]^{1/2}\mathbb{P}(\tau_R \le T)^{1/2} \cr
&\leq & \mathbb{E}\Big[\sup _{t \in[0, T]}\left\|\tilde {x}^{(\varepsilon,\delta)}-\hat{x}^{(\varepsilon,\delta)}\right\|_{\alpha,[0,t]}^4 \Big]^{1/2}\times R^{-1}\sqrt{ \mathbb{E}\left[\left\|B^H\right\|_{1-\alpha, \infty, T}^2\right]} \cr
&\leq& C R^{-1}\sqrt{ \mathbb{E}\left[\left\|B^H\right\|_{1-\alpha, \infty, T}^2\right]},
\end{eqnarray}
where the final inequality comes from (\ref{3-36}) and \lemref{lem1}.
 
 We now compute the second term in (\ref{b3-40}).
 
Define $D:=\{\|B^H\|_{1-\alpha, \infty, T} \leq R\}$. For $\lambda>0$, we will estimate 
\begin{eqnarray*}\label{b3-39}
\mathbf{I}:=\mathbb{E}\left[\sup _{t \in[0, T]} e^{-\lambda t}\|\tilde x^{(\varepsilon,\delta)}-\hat{x}^{(\varepsilon,\delta)}\|_{\alpha,[0,t]}^2 \mathbf{1}_D\right].
\end{eqnarray*}
It is easy to see that,
\begin{eqnarray*}\label{3-40}
\mathbf{I} &\leq & C \mathbb{E}\left[\sup _{t \in[0, T]}e^{-\lambda t}\left\|\int_0^\cdot\big(f_1( \tilde x_s^{(\varepsilon,\delta)}, \tilde y_s^{(\varepsilon,\delta)})-f_1( \tilde x_{s(\Delta)}^{(\varepsilon,\delta)}, \hat{y}_s^{(\varepsilon,\delta)})\big) d s\right\|_{\alpha,[0,t]}^2 \mathbf{1}_D\right] \cr
&&+C \mathbb{E}\left[\sup _{t \in[0, T]}e^{-\lambda t}\left\|\int_0^\cdot\big(\sigma_1( \tilde x_s^{(\varepsilon,\delta)})-\sigma_1( \hat x_{s}^{(\varepsilon,\delta)})\big) d u_s^{(\varepsilon,\delta)}\right\|_{\alpha,[0,t]}^2 \mathbf{1}_D\right] \cr
&&+C \mathbb{E}\left[\varepsilon\sup _{t \in[0, T]}e^{-\lambda t}\left\|\int_0^\cdot\sigma_1( \tilde x_s^{(\varepsilon,\delta)}) d B_s^H\right\|_{\alpha,[0,t]}^2 \mathbf{1}_D\right] \cr
&=:& I_1+I_2+I_3.
\end{eqnarray*}
Before computing $I_i,i=1,2,3$, we note the following fact. For a measurable function $f:[0,T]\to \mathbb{R}^d$, $A=:\|\int_{0}^{\cdot} fds\|_{{\alpha,[0,t]}}$ needs to be estimated in advance as follows:
\begin{eqnarray}\label{3-39}
A & \leq&\left|\int_0^t f(s) d s\right|+\int_0^t(t-s)^{-1-\alpha} \int_s^t|f(r)| d r d s \cr
& \leq& t^\alpha \int_0^t(t-r)^{-\alpha}|f(r)| d r+C
 \int_0^t(t-r)^{-\alpha}|f(r)| dr \cr
& \leq& C \int_0^t(t-r)^{-\alpha}|f(r)| d r.
\end{eqnarray}

Rearrange $I_1$ as follows, and by  Assumption (\textbf{A2}), (\ref{3-39}),  \lemref{lem1}, we have
\begin{eqnarray}\label{3-41}
I_1 &\leq & C \mathbb{E}\left[\sup _{t \in[0, T]}\left\|\int_0^\cdot\big(f_1( \tilde x_s^{(\varepsilon,\delta)}, \tilde y_s^{(\varepsilon,\delta)})-f_1( \tilde x_{s(\Delta)}^{(\varepsilon,\delta)}, \tilde {y}_s^{(\varepsilon,\delta)})\big) d s\right\|_{\alpha,[0,t]}^2 \mathbf{1}_D\right] \cr
&&+C \mathbb{E}\left[\sup _{t \in[0, T]}\left\|\int_0^\cdot\big(f_1( \tilde x_{s(\Delta)}^{(\varepsilon,\delta)}, \tilde {y}_s^{(\varepsilon,\delta)})-f_1( \tilde x_{s(\Delta)}^{(\varepsilon,\delta)}, \hat{y}_s^{(\varepsilon,\delta)})\big) d s\right\|_{\alpha,[0,t]}^2 \mathbf{1}_D\right] \cr
&\leq & C\Big[\sup _{t \in[0, T]}\int_{0}^{t} (t-s)^{-2\alpha}ds\Big]\times \int_0^T\mathbb{E}\left[\big|f_1( \tilde x_s^{(\varepsilon,\delta)}, \tilde y_s^{(\varepsilon,\delta)})-f_1( \tilde x_{s(\Delta)}^{(\varepsilon,\delta)}, \tilde {y}_s^{(\varepsilon,\delta)})\big|^2 \mathbf{1}_D\right]d s \cr
&&+C\Big[\sup _{t \in[0, T]}\int_{0}^{t} (t-s)^{-2\alpha}ds\Big]\times \int_0^T\mathbb{E}\left[\big|f_1( \tilde x_{s(\Delta)}^{(\varepsilon,\delta)}, \tilde {y}_s^{(\varepsilon,\delta)})-f_1( \tilde x_{s(\Delta)}^{(\varepsilon,\delta)}, \hat{y}_s^{(\varepsilon,\delta)})\big|^2\right]d s \cr
&\leq & C \int_0^T\mathbb{E}\left[\big|f_1( \tilde x_s^{(\varepsilon,\delta)}, \tilde y_s^{(\varepsilon,\delta)})-f_1( \tilde x_{s(\Delta)}^{(\varepsilon,\delta)}, \tilde {y}_s^{(\varepsilon,\delta)})\big|^2 \mathbf{1}_D\right]d s \cr
&&+C \int_0^T\mathbb{E}\left[\big|f_1( \tilde x_{s(\Delta)}^{(\varepsilon,\delta)}, \tilde {y}_s^{(\varepsilon,\delta)})-f_1( \tilde x_{s(\Delta)}^{(\varepsilon,\delta)}, \hat{y}_s^{(\varepsilon,\delta)})\big|^2\mathbf{1}_D\right]d s \cr
&\leq & C \int_0^T\mathbb{E}\left[\big|\tilde x_s^{(\varepsilon,\delta)}-\tilde x_{s(\Delta)}^{(\varepsilon,\delta)}\big|^2+\big|\tilde y_s^{(\varepsilon,\delta)}-\hat y_s^{(\varepsilon,\delta)}\big|^2\mathbf{1}_D\right]d s \cr
&\leq &C \Delta,
\end{eqnarray}
where the constant $C$ is independent of $\varepsilon, \delta,\Delta$ which may change from line to line. Note that $2\alpha<1$. The final inequality comes from \lemref{lem2}  and the fact that
\begin{eqnarray*}\label{3-42}
\int_{0}^{T}\mathbb{E}\big[|\tilde y_{t}^{(\varepsilon,\delta)}-\hat y_{t}^{(\varepsilon,\delta)}|^2\big]dt \leq C \Delta,
\end{eqnarray*}
which is a slight extension of   \cite[Lemma 4.4]{2020Large}.

To estimate $I_2$, we firstly give the following estimate: $${B}:=\left\|\int_0^\cdot f(r) d u^{(\varepsilon,\delta)}_r\right\|_{\alpha,[0,t]},$$ 
where $f$ is a measurable function.
In the similar way to the proof of \lemref{lem1}, we have
\begin{eqnarray*}\label{3-43}
{B} &\leq& C \|u^{(\varepsilon,\delta)}\|_{1-\alpha,\infty,T} \int_0^t\left((t-r)^{-2 \alpha}+r^{-\alpha}\right)\left(|f(r)|+\int_0^r \frac{|f(r)-f(q)|}{(r-q)^{1+\alpha}} d q\right) d r\cr
&\leq& C \|u^{(\varepsilon,\delta)}\|_{\mathcal{H}^{H,d_1}} \int_0^t\left((t-r)^{-2 \alpha}+r^{-\alpha}\right)\left(|f(r)|+\int_0^r \frac{|f(r)-f(q)|}{(r-q)^{1+\alpha}} d q\right) d r\cr
&\leq& C \int_0^t\left((t-r)^{-2 \alpha}+r^{-\alpha}\right)\|f\|_{\alpha,[0,r]} d r.
\end{eqnarray*}
Here, we used $\|u^{(\varepsilon,\delta)}\|_{\mathcal{H}^{H,d_1}}\le\sqrt{2N} $.

By \cite[Lemma 7.1]{2002Rascanu}, we have that
\begin{eqnarray*}
\mid \sigma\left( x_1\right)-\sigma\left( x_2\right) -\sigma\left( x_3\right)+\sigma\left( x_4\right) \mid 
\leq  C\left|x_1-x_2-x_3+x_4\right|+C\left|x_1-x_3\right|\left(\left|x_1-x_2\right|+\left|x_3-x_4\right|\right) .
\end{eqnarray*}
Now we estimate $I_2$ as follows:
\begin{eqnarray*}\label{3-45}
I_2 &\le &C \mathbb{E}\left[\sup _{t \in[0, T]}\left|\int_0^t e^{-\lambda t}\left[(t-r)^{-2 \alpha}+r^{-\alpha}\right]\big\|\sigma_1( \tilde {x}^{(\varepsilon,\delta)})-\sigma_1( \hat{x}^{(\varepsilon,\delta)})\big\|_{\alpha,[0,r]} \mathbf{1}_D d r\right|^2\right] \cr
&\leq & C \mathbb{E}\left[\sup _{t \in[0, T]} \mid \int_0^t e^{-\lambda t}\left[(t-r)^{-2 \alpha}+r^{-\alpha}\right]\right.\cr
&&\left.\times\left.\big(1+\Delta(\tilde{x}_r^{(\varepsilon,\delta)})+\Delta(\hat{x}_r^{(\varepsilon,\delta)})\big)\big\|\tilde {x}^{(\varepsilon,\delta)}-\hat{x}^{(\varepsilon,\delta)}\big\|_{\alpha,[0,r]} 1_D d r\right|^2\right],
\end{eqnarray*}
where $\Delta(\tilde{x}_r^{(\varepsilon,\delta)})=\int_0^r \frac{|\tilde{x}_r^{(\varepsilon,\delta)}-\tilde{x}_q^{(\varepsilon,\delta)}|}{(r-q)^{1+\alpha}} d q$ and $\Delta(\hat{x}_r^{(\varepsilon,\delta)})=\int_0^r \frac{|\hat{x}_r^{(\varepsilon,\delta)}-\hat{x}_q^{(\varepsilon,\delta)}|}{(r-q)^{1+\alpha}} d q$. These above two parts can be dominated by $\|\tilde x^{(\varepsilon,\delta)}\|_{\alpha, \infty}$ and $\|\hat x^{(\varepsilon,\delta)}\|_{\alpha, \infty}$, which are in turn dominated by $C=C(R,N)>0$ independent of $\varepsilon,\delta,\Delta$. 
Then, from  (\ref{3-3}) and (\ref{3-10}), we can see that
\begin{eqnarray}\label{3-47}
I_2 &\le &C \lambda^{2 \alpha-1} \mathbb{E}\big[\sup _{t \in[0, T]} e^{-\lambda t}\|\tilde{x}^{(\varepsilon,\delta)}-\hat{x}^{(\varepsilon,\delta)}\|_{\alpha,[0,t]}^2 \mathbf{1}_D\big].
\end{eqnarray}

For the third term $I_3$, by the \lemref{lem1}, we have
\begin{eqnarray}\label{3-48}
I_3 &\le &\varepsilon C \mathbb{E}\left[\sup _{t \in[0, T]}e^{-\lambda t}\left\|\int_0^\cdot\sigma_1( \tilde x_s^{(\varepsilon,\delta)}) d B_s^H\right\|_{\alpha,[0,t]}^2 \mathbf{1}_D\right]\le  \varepsilon C
\end{eqnarray}
for some constant $C=C(R,N)>0$ independent of $\varepsilon,\delta,\Delta$.

Combining  (\ref{3-41}), (\ref{3-47}) and  (\ref{3-48}), we have
\begin{eqnarray*}\label{3-49}
\mathbb{E}\left[\sup _{t \in[0, T]} e^{-\lambda t}\|\tilde x^{(\varepsilon,\delta)}-\hat{x}^{(\varepsilon,\delta)}\|_{\alpha,[0,t]}^2 \mathbf{1}_D\right]\le C\Delta ,
\end{eqnarray*}
which immediately implies that
\begin{eqnarray*}\label{3-50}
\mathbb{E}\left[\sup _{t \in[0, T]} \|\tilde x_t^{(\varepsilon,\delta)}-\hat{x}_t^{(\varepsilon,\delta)}\|_{\alpha,[0,t]}^2 \mathbf{1}_D\right]\le C\Delta.
\end{eqnarray*}
With aid of the (\ref{b3-41}), we have
\begin{eqnarray*}\label{3-51}
\mathbb{E}\big[\sup _{t \in[0, T]}\|\tilde x^{(\varepsilon,\delta)}-\hat{x}^{(\varepsilon,\delta)}\|_{\alpha,[0,t]}^2\big] \leq C \Delta+C^{\prime}R^{-1}
\end{eqnarray*}
for certain constants $C=C(R,N)>0$ independent of $\varepsilon,\delta,\Delta$, and $C^{\prime}$  independent of $\varepsilon,\delta,\Delta,R$.

For $(u^{(\varepsilon, \delta)}, v^{(\varepsilon, \delta)})\in \mathcal{A}^{N}_b$, we construct  the following differential equation:
\begin{eqnarray*}
d\tilde {x}^{\varepsilon}_t =&\bar f_{1}(\tilde {x}^\varepsilon_{t})dt + \sigma_{1}(\tilde {x}^{\varepsilon}_t)du^{(\varepsilon, \delta)}_t
\end{eqnarray*}
with $\tilde {x}^{\varepsilon}_0= {x}_0$. In other word, $\tilde {x}^{\varepsilon}=\mathcal{G}^{0}(u^{(\varepsilon, \delta)},v^{(\varepsilon, \delta)})$.
We will show that 
\begin{eqnarray*}\label{3-52}
\mathbb{E}\big[\sup _{t \in[0, T]}\|\hat x^{(\varepsilon,\delta)}-\tilde{x}^{\varepsilon}\|_{\alpha,[0,t]}^2\big] \leq C \Delta,
\end{eqnarray*}
where $C>0$ is a constant independent of $\varepsilon, \delta,\Delta$. So, we have
\begin{eqnarray*}\label{3-53}
\mathbb{E}\big[\sup _{t \in[0, T]}e^{-\lambda t}\|\hat x^{(\varepsilon,\delta)}-\tilde{x}^{\varepsilon}\|_{\alpha,[0,t]}^2\big]  &\leq & C \mathbb{E}\left[\sup _{t \in[0, T]}e^{-\lambda t}\bigg\|\int_0^\cdot\big(f_1( \tilde x_{s(\Delta)}^{(\varepsilon,\delta)}, \hat{y}_s^{(\varepsilon,\delta)})-\bar f_1( \tilde x_{s(\Delta)}^{(\varepsilon,\delta)})\big) d s\bigg\|_{\alpha,[0,t]}^2 \right]\cr
&&+C \mathbb{E}\left[\sup _{t \in[0, T]}e^{-\lambda t}\bigg\|\int_0^\cdot\big(\bar f_1( \tilde x_{s(\Delta)}^{(\varepsilon,\delta)})-\bar f_1( \tilde x_{s}^{\varepsilon})\big) d s\bigg\|_{\alpha,[0,t]}^2 \right]\cr
&&+C \mathbb{E}\left[\sup _{t \in[0, T]}e^{-\lambda t}\bigg\|\int_0^\cdot\big(\sigma_1( \hat x_s^{(\varepsilon,\delta)})-\sigma_1( \tilde x_{s}^{\varepsilon})\big) d u_s^\varepsilon\bigg\|_{\alpha,[0,t]}^2 \right] \cr
&=:& J_1+J_2+J_3.
\end{eqnarray*}
Firstly, we estimate $J_1$ as follows,
\begin{eqnarray*}\label{3-54}
J_1 &\leq & C \mathbb{E}\bigg[\sup _{t \in[0, T]}\bigg|\sum_{k=0}^{\left\lfloor\frac{t}{\Delta}\right\rfloor-1} \int_{k \Delta}^{(k+1) \Delta}\big(f_1(\tilde x_{k \Delta}^{(\varepsilon,\delta)}, \hat{y}_s^{(\varepsilon,\delta)})-\bar{f}_1( \tilde x_{k \Delta}^{(\varepsilon,\delta)})\big) d s\bigg|^2\bigg] \cr
&&+C \mathbb{E}\bigg[\sup _{t \in[0, T]}\bigg|\int_{\left\lfloor\frac{t}{\Delta}\right\rfloor \Delta}^t\big(f_1( \tilde x_{s(\Delta)}^{(\varepsilon,\delta)}, \hat{y}_s^{(\varepsilon,\delta)})-\bar{f}_1( \tilde x_{s(\Delta)}^{(\varepsilon,\delta)})\big) d s\bigg|^2\bigg] \cr
&&+C \mathbb{E}\bigg[\sup _{t \in[0, T]}\bigg(\int_0^t \frac{\left|\int_s^t\big(f_1( \tilde x_{r(\Delta)}^{(\varepsilon,\delta)}, \hat{y}_r^{(\varepsilon,\delta)})-\bar{f}_1( \tilde x_{r(\Delta)}^{(\varepsilon,\delta)}\big)) d r\right|}{(t-s)^{1+\alpha}} d s\bigg)^2\bigg] \cr
&=:&  \sum_{i=1}^3 {J}_1^i .
\end{eqnarray*}
We compute  terms $\sum_{i=1}^2 {J}_1^i$ as follows,
\begin{eqnarray}\label{3-55}
\sum_{i=1}^2 {J}_1^i &\leq & C \mathbb{E}\big[\sup _{t \in[0, T]}\left\lfloor\frac{t}{\Delta}\right\rfloor \sum_{k=0}^{\left\lfloor\frac{t}{\Delta}\right\rfloor-1}\big|\int_{k \Delta}^{(k+1) \Delta}\big(f_1(\tilde x_{k \Delta}^{(\varepsilon,\delta)}, \hat{y}_s^{(\varepsilon,\delta)})-\bar{f}_1( \tilde x_{k \Delta}^{(\varepsilon,\delta)})\big) d s\big|^2\big] \cr
&&+C_T \Delta^2 \cr
&\leq & \frac{C}{\Delta^2} \max _{0 \leq k \leq\left\lfloor\frac{T}{\Delta}\right\rfloor-1} \mathbb{E}\big[\big|\int_{k \Delta}^{(k+1) \Delta}\big(f_1( \tilde x_{k \Delta}^{(\varepsilon,\delta)})-\bar{f}_1( \tilde x_{k \Delta}^{(\varepsilon,\delta)})\big) d s\big|^2\big] +C_T \Delta^2 \cr
&\leq & C \frac{\delta^2}{\Delta^2} \max _{0 \leq k \leq\left\lfloor\frac{T}{\Delta}\right\rfloor-1} \int_0^{\frac{\Delta}{\delta}} \int_\zeta^{\frac{\Delta}{\delta}} \mathcal{J}_k(s, \zeta) d s d \zeta+C_T \Delta^2,
\end{eqnarray}
where $0 \leq \zeta \leq s \leq \frac{\Delta}{\delta}$ and
\begin{eqnarray}\label{3-56}
\mathcal{J}_k(s, \zeta)&=&\mathbb{E}\big[\big\langle f_1( \tilde x_{k \Delta}^{(\varepsilon,\delta)}, \hat{y}_{s \varepsilon+k \Delta}^{(\varepsilon,\delta)})-\bar{f}_1( \tilde x_{k \Delta}^{(\varepsilon,\delta)}),f_1( \tilde x_{k \Delta}^{(\varepsilon,\delta)}, \hat{y}_{\zeta \varepsilon+k \Delta}^{(\varepsilon,\delta)})-\bar{f}_1( \tilde x_{k \Delta}^{(\varepsilon,\delta)})\big\rangle\big].
\end{eqnarray}
It is known that \begin{eqnarray}\label{er}
\mathcal{J}_k(s, \zeta) \leq C e^{-\frac{\beta_1}{2}(s-\zeta)},
\end{eqnarray}
where $\beta_1$ is  in (\textbf{A3}), whose proof could  refer to  \cite[Appendix B]{2023Pei}.

We also compute $J_1^3$ as follows:
\begin{eqnarray}\label{3-57}
J_1^3 
&\leq & C \mathbb{E}\big[\sup _{t \in[0, T]}\int_0^t{(t-s)^{-\frac{1}{2}-\alpha}}ds\times\sup _{t \in[0, T]}\int_0^t \frac{\big|\int_s^t\big(f_1( \tilde x_{r(\Delta)}^{(\varepsilon,\delta)}, \hat{y}_r^{(\varepsilon,\delta)})-\bar{f}_1( \tilde x_{r(\Delta)}^{(\varepsilon,\delta)}\big)) d r\big|^2}{(t-s)^{\frac{3}{2}+\alpha}} d s\big] \cr
&\leq & C \mathbb{E} \big[\sup _{t \in[0, T]}\int_0^t \frac{\big|\int_s^t\big(f_1( \tilde x_{r(\Delta)}^{(\varepsilon,\delta)}, \hat{y}_r^{(\varepsilon,\delta)})-\bar{f}_1( \tilde x_{r(\Delta)}^{(\varepsilon,\delta)}\big)) d r\big|^2}{(t-s)^{\frac{3}{2}+\alpha}}\mathbf{1}_{\ell^c} d s\big] \cr
&&+ C  \mathbb{E}\big[\sup _{t \in[0, T]}\int_0^t \frac{\big|\int_s^t\big(f_1( \tilde x_{r(\Delta)}^{(\varepsilon,\delta)}, \hat{y}_r^{(\varepsilon,\delta)})-\bar{f}_1( \tilde x_{r(\Delta)}^{(\varepsilon,\delta)}\big)) d r\big|^2}{(t-s)^{\frac{3}{2}+\alpha}} \mathbf{1}_\ell d s\big] \cr
&=:& {J}_1^{31}+ {J}_1^{32}.
\end{eqnarray}
where $\ell:=\left\{t-s<2 \Delta\right\}$ and $\ell^c:=\left\{t-s \geq2 \Delta\right\}$.

Then, we compute the term ${J}_1^{31}$,
\begin{eqnarray}\label{3-58}
{J}_1^{31} 
&\leq &C \mathbb{E} \bigg[\sup _{t \in[0, T]}\int_0^t \frac{\big|\int_s^t\big(f_1( \tilde x_{r(\Delta)}^{(\varepsilon,\delta)}, \hat{y}_r^{(\varepsilon,\delta)})-\bar{f}_1( \tilde x_{r(\Delta)}^{(\varepsilon,\delta)}\big)) d r\big|^2}{(t-s)^{\frac{3}{2}+\alpha}}\mathbf{1}_{\ell^c} d s\bigg] \cr
&\le&C \mathbb{E}\bigg[\sup _{t \in[0, T]} \int_0^t \frac{\left|\int_s^{\left\lfloor\frac{s}{\Delta}\right\rfloor \Delta+1}\big(f_1( \tilde x_{r(\Delta)}^{(\varepsilon,\delta)}, \hat{y}_r^{(\varepsilon,\delta)})-\bar{f}_1( \tilde x_{r(\Delta)}^{(\varepsilon,\delta)})\big) d r\right|^2}{(t-s)^{\frac{3}{2}+\alpha}} \mathbf{1}_{\ell^c} d s\bigg] \cr
&&+C \mathbb{E}\bigg[\sup _{t \in[0, T]} \int_0^t \frac{\left|\int_{\left\lfloor\frac{t}{\Delta}\right\rfloor \Delta}^t\big(f_1( \tilde x_{r(\Delta)}^{(\varepsilon,\delta)}, \hat{y}_r^{(\varepsilon,\delta)})-\bar{f}_1( \tilde x_{r(\Delta)}^{(\varepsilon,\delta)})\big) d r\right|^2}{(t-s)^{\frac{3}{2}+\alpha}} \mathbf{1}_{\ell^c} d s\bigg] \cr
&&+C \mathbb{E}\bigg[\sup _{t \in[0, T]} \int_0^t \frac{\left(\left\lfloor\frac{t}{\Delta}\right\rfloor-\left\lfloor\frac{s}{\Delta}\right\rfloor-1\right)}{(t-s)^{\frac{3}{2}+\alpha}}\times \sum_{k=\left\lfloor\frac{s}{\Delta}\right\rfloor+1}^{\left\lfloor\frac{t}{\Delta}\right\rfloor-1}\bigg|\int_{k \Delta}^{(k+1) \Delta}\big(f_1(\tilde x_{k \Delta}^{(\varepsilon,\delta)}, \hat{y}_r^{(\varepsilon,\delta)})-\bar{f}_1( \tilde x_{k \Delta}^{(\varepsilon,\delta)})\big) d r\bigg|^2 \mathbf{1}_{\ell^c} d s\bigg] \cr
&\leq& C \sup _{t \in[0, T]}\bigg(\int_0^t(t-s)^{-\frac{1}{2}-\alpha}((\left\lfloor\frac{s}{\Delta}\right\rfloor+1) \Delta-s) d s\bigg)+C \sup _{t \in[0, T]}\bigg(\int_0^t(t-s)^{-\frac{1}{2}-\alpha}(t-\left\lfloor\frac{t}{\Delta}\right\rfloor \Delta)  d s\bigg) \cr
&&+C \Delta^{-1} \mathbb{E}\bigg[\sup _{t \in[0, T]} \int_0^t(t-s)^{-\frac{1}{2}-\alpha}\times \sum_{k=\left\lfloor\frac{s}{\Delta}\right\rfloor+1}^{\left\lfloor\frac{t}{\Delta}\right\rfloor-1}\big|\int_{k \Delta}^{(k+1) \Delta}\big(f_1( \tilde x_{k \Delta}^{(\varepsilon,\delta)}, \hat{y}_r^{(\varepsilon,\delta)})-\bar{f}_1( \tilde x_{k \Delta}^{(\varepsilon,\delta)})\big) d r\big|^2 \mathbf{1}_{\ell^c} d s\bigg] \cr
&\leq& C \Delta+C \frac{\delta^2}{\Delta^2} \max _{0 \leq k \leq\left\lfloor\frac{T}{\Delta}\right\rfloor-1} \int_0^{\frac{\Delta}{\delta}} \int_\zeta^{\frac{\Delta}{\delta}} \mathcal{J}_k(s, \zeta) d s d \zeta .
\end{eqnarray}
For the term ${J}_1^{32}$, by  Assumption (\textbf{A2}), \lemref{lem1} and the fact that $t-s\leq 2 \Delta$, we have
\begin{eqnarray}\label{3-59}
{J}_1^{32}
&\leq & C \mathbb{E}\left[\sup _{t \in[\Delta, T]} \int_0^{t(\Delta)-\Delta} \frac{\big|\int_s^t\big(f_1(\tilde x_{r(\Delta)}^{(\varepsilon,\delta)}, \hat{y}_r^{(\varepsilon,\delta)})-\bar{f}_1(\tilde x_{r(\Delta)}^{(\varepsilon,\delta)})\big) d r\big|^2}{(t-s)^{\frac{3}{2}+\alpha}} \mathbf{1}_{ \ell} d s\right] \cr
&&+C \mathbb{E}\left[\sup _{t \in[\Delta, T]} \int_{t(\Delta)-\Delta}^t \frac{\big|\int_s^t\big(f_1(\tilde x_{r(\Delta)}^{(\varepsilon,\delta)}, \hat{y}_r^{(\varepsilon,\delta)})-\bar{f}_1(\tilde x_{r(\Delta)}^{(\varepsilon,\delta)})\big) d r\big|^2}{(t-s)^{\frac{3}{2}+\alpha}}\mathbf{1}_{ \ell} d s\right] \cr
&&+C\mathbb{E}\left[\sup _{t \in[0, \Delta]} \int_0^t \frac{\left|\int_s^t\big(f_1(\tilde x_{r(\Delta)}^{(\varepsilon,\delta)}, \hat{y}_r^{(\varepsilon,\delta)})-\bar{f}_1(\tilde x_{r(\Delta)}^{(\varepsilon,\delta)})\big) d r\right|^2}{(t-s)^{\frac{3}{2}+\alpha}} 1_{\ell} d s\right] \cr
&\leq & C \Delta \sup _{t \in[\Delta, T]}\big(\int_0^{t(\Delta)-\Delta}(t-s)^{-\frac{1}{2}-\alpha}  ds\big) +C \sup _{t \in[\Delta, T]}\big(\int_{t(\Delta)-\Delta}^t(t-s)^{\frac{1}{2}-\alpha} \mathbf{1}_{ \ell} ds\big) \cr
&&+C \sup _{t \in[0, \Delta]}\big(\int_0^t(t-s)^{\frac{1}{2}-\alpha} \mathbf{1}_{ \ell}  ds\big)\cr
&\leq & C \Delta.
\end{eqnarray}

Combining  (\ref{3-55})--(\ref{3-59}), we have
\begin{eqnarray}\label{3-60}
J_1
&\leq & C \Delta+C \frac{\delta^2}{\Delta^2} \max _{0 \leq k \leq\left\lfloor\frac{T}{\Delta}\right\rfloor-1} \int_0^{\frac{\Delta}{\delta}} \int_\zeta^{\frac{\Delta}{\delta}} \mathcal{J}_k(s, \zeta) d s d \zeta\cr
& \leq& C \frac{\delta^2}{\Delta^2} \max _{0 \leq k \leq\left\lfloor\frac{T}{\Delta}\right\rfloor-1} \int_0^{\frac{\Delta}{\delta}} \int_\zeta^{\frac{\Delta}{\delta}} e^{-\frac{\beta_1}{2}(s-\zeta)} d s d \zeta+C \Delta \cr
& \leq &C \frac{\delta^2}{\Delta^2}\left(\frac{2}{\beta_1} \frac{\Delta}{\delta}-\frac{4}{\beta_1^2}+e^{\frac{-\beta_1}{2} \frac{\Delta}{\delta}}\right)+C\Delta \cr
& \leq& C\left(\delta \Delta^{-1}+\Delta\right),
\end{eqnarray}
where the final inequality comes from that choosing $\Delta=\Delta(\delta)$ such that $\frac{\Delta}{\delta}$ is sufficiently large.

Subsequently, it proceeds to estimate $J_2$,
\begin{eqnarray}\label{3-61}
J_2
&\leq &  C \mathbb{E}\bigg[\sup _{t \in[0, T]} e^{-\lambda t}\big\|\int_0^\cdot\big(\bar{f}_1( \tilde {x}_{s(\Delta)}^{(\varepsilon,\delta)})-\bar{f}_1( \tilde {x}_s^{(\varepsilon,\delta)})\big) d s\big\|_{\alpha,[0,t]}^2 \bigg] \cr
&&+C \mathbb{E}\bigg[\sup _{t \in[0, T]} e^{-\lambda t}\big\|\int_0^\cdot\big(\bar{f}_1(\tilde {x}_s^{(\varepsilon,\delta)})-\bar{f}_1(\hat {x}_s^{(\varepsilon,\delta)})\big) d s\big\|_{\alpha,[0,t]}^2 \bigg] \cr
&&+C \mathbb{E}\bigg[\sup _{t \in[0, T]} e^{-\lambda t}\big\|\int_0^\cdot\big(\bar{f}_1( \hat {x}_s^{(\varepsilon,\delta)})-\bar{f}_1( \tilde {x}_s^{(\varepsilon)})\big) d s\big\|_{\alpha,[0,t]}^2 \bigg]\cr
&=:& J_2^1+J_2^2+J_2^3.
\end{eqnarray}
Taking same manner to (\ref{3-41}), with   Assumption (\textbf{A2}), \lemref{lem1} and (\ref{3-51}), we have
\begin{eqnarray}\label{3-62}
\sum_{i=1}^{2} J_2^i 
&\leq & C \int_0^T\mathbb{E}\left[\big|\bar{f}_1( \tilde x_s^{(\varepsilon,\delta)})-\bar{f}_1( \tilde x_{s(\Delta)}^{(\varepsilon,\delta)})\big|^2 \right]d s +C \int_0^T\mathbb{E}\left[\big|\bar{f}_1( \tilde x_{s(\Delta)}^{(\varepsilon,\delta)})-\bar{f}_1( \hat x_{s}^{(\varepsilon,\delta)})\big|^2\right]d s \cr
&\leq & C \int_0^T\mathbb{E}\left[\big|\tilde x_s^{(\varepsilon,\delta)}-\tilde x_{s(\Delta)}^{(\varepsilon,\delta)}\big|^2+\big|\hat x_s^{(\varepsilon,\delta)}-\tilde x_{s}^{(\varepsilon,\delta)}\big|^2\right]d s \cr
&\leq &C \Delta+C^{\prime}R^{-1}
\end{eqnarray}
for certain constants $C=C(R,N)>0$ independent of $\varepsilon,\delta,\Delta$, and $C^{\prime}$  independent of $\varepsilon,\delta,\Delta,R$.

By the Lipschitz continuity of $\bar f_1$, we also have
\begin{eqnarray}\label{3-63}
J_2^3 
&\leq  & C \mathbb{E}\left[\sup _{t \in[0, T]} e^{-\lambda t} \int_0^t(t-s)^{-2 \alpha} \big| \bar{f}_1( \hat{x}_s^{(\varepsilon,\delta)})-\bar{f}_1( \tilde{x}_s^{(\varepsilon)})\big|^2  d s\right] \cr
& \leq& C\mathbb{E}\left[\sup _{t \in[0, T]} \int_0^t e^{-\lambda(t-s)}(t-s)^{-2 \alpha} e^{-\lambda s}\big|\hat{x}_s^{(\varepsilon,\delta)}-\tilde{x}_s^{(\varepsilon)}\big|^2  d s\right] \cr
& \leq& C\mathbb{E}\left[\sup _{t \in[0, T]} e^{-\lambda t}\big\|\hat{x}^{(\varepsilon,\delta)}-\tilde{x}^{(\varepsilon)}\big\|_{\alpha,[0,t]}^2 \right] \sup _{t \in[0, T]} \int_0^t e^{-\lambda(t-r)}(t-r)^{-2 \alpha} d r \cr
& \leq& C \lambda^{2 \alpha-1} \mathbb{E}\big[\sup_{t \in[0, T]}e^{-\lambda t} \big\|\hat{x}^{(\varepsilon,\delta)}-\tilde{x}^{(\varepsilon)}\big\|_{\alpha,[0,t]}^2 \big].
\end{eqnarray}
For the term $J_3$, take same manner in (\ref{3-45}), we have
\begin{eqnarray*}\label{3-64}
J_3 
&\leq & C \mathbb{E}\left[\sup _{t \in[0, T]}\left|\int_0^t e^{-\lambda t}\left[(t-r)^{-2 \alpha}+r^{-\alpha}\right]\left\|\sigma_1(  \hat {x}^{(\varepsilon,\delta)})-\sigma_1(  \tilde {x}^{(\varepsilon)})\right\|_{\alpha,[0,s]} d r\right|^2\right] \cr
&\leq & C\mathbb{E}\bigg[\sup _{t \in[0, T]} \mid \int_0^t e^{-\lambda t}\big[(t-r)^{-2 \alpha}+r^{-\alpha}\big]\bigg.\cr
&&\bigg.\times\left.\left(1+\Delta(  \hat {x}_r^{(\varepsilon,\delta)})+\Delta( \tilde {x}_r^{(\varepsilon)})\right)\| \hat {x}^{(\varepsilon,\delta)}- \tilde {x}^{(\varepsilon)}\|_{\alpha,[0,r]}  d r\right|^2\bigg],
\end{eqnarray*}
where  $\Delta(\hat{x}_r^{(\varepsilon,\delta)})=\int_0^r \frac{|\hat{x}_r^{(\varepsilon,\delta)}-\hat{x}_q^{(\varepsilon,\delta)}|}{(r-q)^{1+\alpha}} d q$ and $\Delta(\tilde{x}_r^{(\varepsilon)})=\int_0^r \frac{|\tilde{x}_r^{(\varepsilon)}-\tilde{x}_q^{(\varepsilon)}|}{(r-q)^{1+\alpha}} d q$. These above two parts can be dominated by $\|\hat x^{(\varepsilon,\delta)}\|_{\alpha, \infty}$ and $\|\tilde x^{\varepsilon}\|_{\alpha, \infty}$, which are in turn dominated by $C>0$ independent of $\varepsilon,\delta,\Delta$.

Therefore, 
\begin{eqnarray}\label{3-66}
J_3 
& \leq& C \lambda^{2 \alpha-1} \mathbb{E}\big[\sup_{t \in[0, T]} \big\|\hat{x}_s^{(\varepsilon,\delta)}-\tilde{x}_s^{(\varepsilon)}\big\|_\alpha^2 \big].
\end{eqnarray}
Combine (\ref{3-60})--(\ref{3-66}), we obtain
\begin{eqnarray*}\label{3-67}
\mathbb{E}\big[\sup _{t \in[0, T]}\|\hat x^{(\varepsilon,\delta)}-\tilde{x}^{\varepsilon}\|_{\alpha,[0,t]}^2\big] 
& \leq& C \lambda^{2 \alpha-1} \mathbb{E}\big[\sup_{t \in[0, T]} \big\|\hat{x}^{(\varepsilon,\delta)}-\tilde{x}^{(\varepsilon)}\big\|_{\alpha,[0,t]}^2 \big]\cr
&&+C\left(\delta \Delta^{-1}+\Delta\right)+C^{\prime}R^{-1}.
\end{eqnarray*}
Taking $\lambda$ large enough such that $C  \lambda^{2 \alpha-1}<1$, we have
\begin{eqnarray*}\label{3-80}
	\mathbb{E}\big[\sup _{t \in[0, T]}\|\hat x^{(\varepsilon,\delta)}-\tilde{x}^{\varepsilon}\|_{\alpha,[0,t]}^2\big] 
	\le C\left(\delta \Delta^{-1}+\Delta\right)+C^{\prime}R^{-1}
\end{eqnarray*}
for certain constants $C=C(R,N)>0$ independent of $\varepsilon,\delta,\Delta$, and $C^{\prime}$  independent of $\varepsilon,\delta,\Delta,R$.

By tha above estimation, and  continuous inclusion $C^{\alpha+\kappa} \subset W_0^{\alpha, \infty} \subset C^{\alpha-\kappa}$ holds for any suffficiently small $\kappa>0$,  we have
\begin{eqnarray*}\label{3-81}
	\lim_{\varepsilon \to 0}\mathbb{E}\big[\|\tilde x^{(\varepsilon,\delta)}-\tilde{x}^{\varepsilon}\|_{\alpha\textrm{-hld}}^2\big]= 0.
\end{eqnarray*}
Precisily, we argued in the following way, we set $\Delta=\varepsilon$, and take $\limsup_{\varepsilon \to 0}$  for every fixed large $R$, then we let $R \to\infty$. Note that $C^{\prime}$ is independent of $R$.

 According to the Step 1, if  $(u^{(\varepsilon, \delta)}, v^{(\varepsilon, \delta)})\in \mathcal{A}^{N}_b$ such that $(u^{(\varepsilon, \delta)}, v^{(\varepsilon, \delta)})$ weakly converges  to $(u, v)$ as $\varepsilon \to 0$, then,  $\{\tilde x^{(\varepsilon)}\}=\{\mathcal{G}^0 (u^{(\varepsilon, \delta)}, v^{(\varepsilon, \delta)})\}$ weakly converges to $\tilde x=\mathcal{G}^0 (u, v)$ in $ C^{\alpha}([0,T],\mathbb{R}^m)$ as $\varepsilon \to 0$. 
Then, by the Portemanteau's theorem \cite[Theorem 13.16]{2020Klenke}, for any bounded Lipschitz functions $f:C^{\alpha}([0, T]; \mathbb{R}^m) \to\mathbb{R}$, we see that as $\varepsilon \to 0$
\begin{eqnarray*}\label{3-82}
|\mathbb{E}[f(\tilde x^{(\varepsilon,\delta)})]-\mathbb{E}[f(\tilde x)]|&\le& |\mathbb{E}[f(\tilde x^{(\varepsilon,\delta)})]-\mathbb{E}[f(\tilde x^{(\varepsilon)})]|+|\mathbb{E}[f(\tilde x^{(\varepsilon)})]-\mathbb{E}[f(\tilde x)]|\cr
&\le& \|f\|_\textrm{Lip}\mathbb{E}[\|\tilde x^{(\varepsilon,\delta)}-\tilde x^{(\varepsilon)}\|_{\alpha\textrm{-hld}}^2]^{\frac{1}{2}}+|\mathbb{E}[f(\tilde x^{(\varepsilon)})]-\mathbb{E}[f(\tilde x)]|\to 0.
\end{eqnarray*} 
 Here, $\|f\|_\textrm{Lip}$ is the Lipschitz constant of $f$. Thus, we have obatined (\ref{step3}).

\textbf{Step 3}. In this step, we prove Laplace upper and lower bound as well as goodness of $I$. By a standard argument, it implies the large deviation in our main theorem.  We write $\mathcal{E}:=  C^{\alpha}([0, T]; \mathbb{R}^m)$ for simplicity. 

First, we will show that the $\{x^{(\varepsilon,\delta)}\}$ satisfies the Laplace principle upper bound on $\mathcal{E}$ with  $I:\mathcal{E} \to [0,\infty]$, that is for  any continuous and bounded function $F : \mathcal{E} \to
\mathbb{R}$, 
\begin{eqnarray}\label{5-5}
\limsup _{\varepsilon \rightarrow 0} -\varepsilon \log \mathbb{E}\left[e^{-F(x^{(\varepsilon,\delta)}) / \varepsilon}\right] \leq\inf _{\phi \in \mathcal{E}}\left[F(\phi)+I(\phi)\right]
\end{eqnarray}
holds.

Without loss we can assume that $\inf _{\phi \in \mathcal{E}}\{F(\phi)+I(\phi)\}<\infty$. Let $\kappa>0$ be arbitrary. So, there exists  $\phi_0\in \mathcal{E}$ such that
$$
F\left(\phi_0\right)+I\left(\phi_0\right) \leq \inf _{\phi \in \mathcal{E}}\{F(\phi)+I(\phi)\}+\frac{\kappa}{2}<\infty .
$$

By the definition of $I:\mathcal{E} \to [0,\infty]$, there exists $(\tilde{u},0)\in \mathcal{H}$ such that
	\begin{eqnarray}\label{5-4}\frac{1}{2}\|\tilde{u}\|_{\mathcal{H}^{H,d_1}}^2 \leq I\left(\phi_0\right)+\frac{\kappa}{2}, \text { and } \phi_0=\mathcal{G}^0(\tilde{u},0).
	\end{eqnarray}
	According to  \propref{prop1}, we can see that
\begin{eqnarray}\label{5-1}
	-\varepsilon \log \mathbb{E}\Big(\exp \Big\{-\frac{F(x^{(\varepsilon,\delta)})}{\varepsilon}\Big\}\Big) & =&-\varepsilon \log \mathbb{E}\Big(\exp \Big\{-\frac{F \circ \mathcal{G}^{(\varepsilon,\delta)}(\sqrt{\varepsilon} B^H,\sqrt{\varepsilon} W)}{\varepsilon}\Big\}\Big) \cr
	& =&\inf _{(u,v) \in \mathcal{A}_b} \mathbb{E}\Big[F \circ \mathcal{G}^{(\varepsilon,\delta)}(\sqrt{\varepsilon} B^H+u,\sqrt{\varepsilon} W+v)+\frac{1}{2}\|(u,v)\|_{\mathcal{H}}^2\Big].
	\end{eqnarray}
 By taking $\limsup_{\varepsilon \rightarrow 0}$ on the both sides of (\ref{5-1}) and applying (\ref{5-4}), we have
\begin{eqnarray*}\label{5-2}
\lefteqn{
\limsup _{\varepsilon \rightarrow 0}-\varepsilon \log \mathbb{E}\Big(\exp \Big\{-\frac{F(x^{(\varepsilon,\delta)})}{\varepsilon}\Big\}\Big)  
}\\
& =&\limsup _{\varepsilon \rightarrow 0} \inf _{(u,v) \in \mathcal{A}_b} \mathbb{E}\Big[F \circ \mathcal{G}^{(\varepsilon,\delta)}(\sqrt{\varepsilon} B^H+u,\sqrt{\varepsilon} W+v)+\frac{1}{2}\|(u,v)\|_{\mathcal{H}}^2\Big] \cr
& \leq& \limsup _{\varepsilon \rightarrow 0} \mathbb{E}\Big[F \circ \mathcal{G}^{(\varepsilon,\delta)}(\sqrt{\varepsilon} B^H+\tilde u,\sqrt{\varepsilon} W+0)+\frac{1}{2}\|\tilde{u}\|_{\mathcal{H}^{H,d_1}}^2\Big] \cr
& \leq &\limsup _{\varepsilon \rightarrow 0} \mathbb{E}\big[F \circ \mathcal{G}^{(\varepsilon,\delta)}(\sqrt{\varepsilon} B^H+\tilde u,\sqrt{\varepsilon} W+0)\big]+I\left(\phi_0\right)+\frac{\kappa}{2}\cr
& = & F\circ \mathcal{G}^{0}(\tilde u,0)+I\left(\phi_0\right)+\frac{\kappa}{2}\cr
&=&F\left(\phi_0\right)+I\left(\phi_0\right)+\frac{\kappa}{2}\cr
 &\le&\inf _{\phi \in \mathcal{E}}\{F(\phi)+I(\phi)\}+\kappa,
\end{eqnarray*}
where the forth line comes from  Step 2.
Since $\kappa>0$ is arbitrary, we have shown the  (\ref{5-5}).

Next, we will show that the $\{x^{(\varepsilon,\delta)}\}$ satisfies the Laplace principle lower bound on $\mathcal{E}$ with  $I:\mathcal{E} \to [0,\infty]$, that is for any continuous and  bounded function $F : \mathcal{E} \to
\mathbb{R}$, 
\begin{eqnarray}\label{5-6}
\liminf _{\varepsilon \rightarrow 0} -\varepsilon \log \mathbb{E}\left[e^{-F(x^{(\varepsilon,\delta)}) / \varepsilon}\right] \geq\inf _{\phi \in \mathcal{E}}\left[F(\phi)+I(\phi)\right]
\end{eqnarray}
holds. 

Before proving  (\ref{5-6}), some prior estimates will be given. Take any $\kappa\in(0,1)$ and fix it for a while. For each $0<\delta<\varepsilon<1$, there exists $(u^{(\varepsilon,\delta)},v^{(\varepsilon,\delta)})\in \mathcal{A}_b$ such that 
\begin{eqnarray*}\label{5-7}
&\inf _{(u,v) \in \mathcal{A}_b} \mathbb{E}\Big[F \circ \mathcal{G}^{(\varepsilon,\delta)}(\sqrt{\varepsilon} B^H+ u,\sqrt{\varepsilon} W+ v)+\frac{1}{2}\|( u, v)\|_{\mathcal{H}}^2\Big] \cr
&\geq \mathbb{E}\Big[F \circ \mathcal{G}^{(\varepsilon,\delta)}(\sqrt{\varepsilon} B^H+ u^{(\varepsilon,\delta)},\sqrt{\varepsilon} W+ v^{(\varepsilon,\delta)})+\frac{1}{2}\|( u^{(\varepsilon,\delta)}, v^{(\varepsilon,\delta)})\|_{\mathcal{H}}^2\Big]-\kappa=:\mathbb{A}.
\end{eqnarray*}
It is easy to see that for all $0<\delta<\varepsilon<1$, $\mathbb{E}\left(\frac{1}{2}\|( u^{(\varepsilon,\delta)}, v^{(\varepsilon,\delta)})\|_{\mathcal{H}}^2\right) \leq 2 \|F\|_{\infty}+1$. Then define the stopping times $\tau_N^{\varepsilon}=\inf \{t \in[0,T]: \frac{1}{2} \int_0^t[|\dot{u}_s^{(\varepsilon,\delta)}|^2+|\{v^{(\varepsilon,\delta)}\}{^\prime_s}|^2]  d s \geq N\} \wedge T$
where $u^{(\varepsilon,\delta)}=\mathcal{\mathcal{K_H}} \dot{u}^{(\varepsilon,\delta)}$ and $\{v^{(\varepsilon,\delta)}\}{^\prime}$ is the time derivative of ${v}^{(\varepsilon,\delta)}$.  Define $({u}^{(\varepsilon,\delta,N)},{{v}^{(\varepsilon,\delta,N)}})\in \mathcal{A}_b^N$ by 
 \begin{eqnarray*}\label{5-21}
(\dot{u}^{(\varepsilon,\delta,N)},\{v^{(\varepsilon,\delta,N)}\}{^\prime})=(\dot{u}^{(\varepsilon,\delta)}\mathbf{1}_{[0,\tau_N^\varepsilon]},\{v^{(\varepsilon,\delta)}\}{^\prime}\mathbf{1}_{[0,\tau_N^\varepsilon]}).
 \end{eqnarray*} 
 Then, we have that
\begin{eqnarray}\label{5-8}
\mathbb{P}\big(( u^{(\varepsilon,\delta,N)}, v^{(\varepsilon,\delta,N)}) \neq ( u^{(\varepsilon,\delta)}, v^{(\varepsilon,\delta)})\big) &\leq& \mathbb{P}\Big(\frac{1}{2} \int_0^t[|\dot{u}_s^{(\varepsilon,\delta)}|^2+|\{v^{(\varepsilon,\delta)}\}{^\prime_s}|^2]  d s \geq N\Big)\cr
&=&\mathbb{P}\Big(\frac{1}{2}\big\|({u}^{(\varepsilon,\delta)},{v}^{(\varepsilon,\delta)})\big\|_{\mathcal{H}}^2 \geq N\Big) \leq \frac{2 \|F\|_{\infty}+1}{N}.
\end{eqnarray}
Note that $\mathbb{A}$ can be rewritten in the following way:
\begin{eqnarray*}\label{5-9}
\mathbb{A}
&\ge& \mathbb{E}\Big[F \circ \mathcal{G}^{(\varepsilon,\delta)}(\sqrt{\varepsilon} B^H+ u^{(\varepsilon,\delta,N)},\sqrt{\varepsilon} W+ v^{(\varepsilon,\delta,N)})+\frac{1}{2}\|( u^{(\varepsilon,\delta,N)}, v^{(\varepsilon,\delta,N)})\|_{\mathcal{H}}^2\Big]-\kappa\cr
&&+ \mathbb{E}\Big[F \circ \mathcal{G}^{(\varepsilon,\delta)}(\sqrt{\varepsilon} B^H+ u^{(\varepsilon,\delta)},\sqrt{\varepsilon} W+ v^{(\varepsilon,\delta)})-F \circ \mathcal{G}^{(\varepsilon,\delta)}(\sqrt{\varepsilon} B^H+ u^{(\varepsilon,\delta,N)},\sqrt{\varepsilon} W+ v^{(\varepsilon,\delta,N)})\Big]\cr
&=:&\mathbb{A}_1+\mathbb{A}_2.
\end{eqnarray*}
By (\ref{5-8}), and choosing $N$ large enough such that $\frac{2 \|F\|_{\infty}(2 \|F\|_{\infty}+\kappa)}{N} \leq \kappa$, we have 
\begin{eqnarray*}\label{5-10}
\mathbb{A}_2
\geq-\kappa.
\end{eqnarray*}
Notice that the integrand of $\mathbb{A}_2$ vanishes on $\{( u^{(\varepsilon,\delta,N)}, v^{(\varepsilon,\delta,N)}) = ( u^{(\varepsilon,\delta)}, v^{(\varepsilon,\delta)})\}$
and that the choice of $N$ is independent of $\varepsilon,\delta$.
Then, we have \begin{eqnarray}\label{5-11}
\mathbb{A}
&\geq& \mathbb{E}\Big[F \circ \mathcal{G}^{(\varepsilon,\delta)}(\sqrt{\varepsilon} B^H+ u^{(\varepsilon,\delta,N)},\sqrt{\varepsilon} W+ v^{(\varepsilon,\delta,N)})+\frac{1}{2}\|( u^{(\varepsilon,\delta,N)}, v^{(\varepsilon,\delta,N)})\|_{\mathcal{H}}^2\Big]-2\kappa.
\end{eqnarray}
Then, we will show (\ref{5-6}). According to  \propref{prop1} and  (\ref{5-11}), we can see that
\begin{eqnarray}\label{5-12}
-\varepsilon \log \mathbb{E}\Big[\exp \Big(-\frac{F(x^{(\varepsilon,\delta)})}{\varepsilon}\Big)\Big] & =&
-\varepsilon \log \mathbb{E}\Big[\exp \Big(-\frac{F \circ \mathcal{G}^{(\varepsilon,\delta)}(\sqrt{\varepsilon} B^H,\sqrt{\varepsilon} W)}{\varepsilon}\Big)\Big] \cr
& =&\inf _{(u,v) \in \mathcal{A}_b} \mathbb{E}\Big[F \circ \mathcal{G}^{(\varepsilon,\delta)}(\sqrt{\varepsilon} B^H+ u,\sqrt{\varepsilon} W+ v)+\frac{1}{2}\|( u, v)\|_{\mathcal{H}}^2\Big]\cr
& \geq&\mathbb{E}\Big[F \circ \mathcal{G}^{(\varepsilon,\delta)}(\sqrt{\varepsilon} B^H+ u^{(\varepsilon,\delta,N)},\sqrt{\varepsilon} W+ v^{(\varepsilon,\delta,N)})\Big.\cr
&&\qquad\qquad\qquad+\Big.\frac{1}{2}\|( u^{(\varepsilon,\delta,N)}, v^{(\varepsilon,\delta,N)})\|_{\mathcal{H}}^2\Big]-2\kappa.
\end{eqnarray}

There exists $\{\varepsilon_k,\delta_k\}_{k=1}^{\infty}$ such that $0<\delta_k<\varepsilon_k$, $\lim\limits_{k\to \infty}\varepsilon_k= \lim\limits_{k\to \infty} \frac{\delta_k}{\varepsilon_k}=0$ and
\begin{eqnarray*}\label{5-13}
\liminf_{\varepsilon \rightarrow 0}-\varepsilon \log \mathbb{E}\Big[\exp \Big(-\frac{F(x^{(\varepsilon,\delta)})}{\varepsilon}\Big)\Big]
&=& \lim_{k \rightarrow \infty}-\varepsilon_k \log \mathbb{E}\Big[\exp \Big(-\frac{F(x^{(\varepsilon_k,\delta_k)})}{\varepsilon_k}\Big)\Big].
\end{eqnarray*}
The laws of the  family $\{{u}^{(\varepsilon_k,\delta_k,N)},{{v}^{(\varepsilon_k,\delta_k,N)}}\}_{k=1}^{\infty}\subset \mathcal{A}_b^N$ are tight since $S_N$ is compact. So there exists a subsequence (denoted by the same symbol) such that  $\{{u}^{(\varepsilon_k,\delta_k,N)},{{v}^{(\varepsilon_k,\delta_k,N)}}\}_{k=1}^{\infty}$ weakly converges to some $({u}^{\infty},{{v}^{\infty}})$ as $k \to \infty$. Then 
\begin{eqnarray*}\label{5-14}
\liminf_{\varepsilon \rightarrow 0}-\varepsilon \log \mathbb{E}\Big[\exp \Big(-\frac{F(x^{(\varepsilon,\delta)})}{\varepsilon}\Big)\Big] & \geq&\liminf_{k \rightarrow \infty}\mathbb{E}\Big[F \circ \mathcal{G}^{(\varepsilon_k,\delta_k)}(\sqrt{\varepsilon_k} B^H+ u^{(\varepsilon_k,\delta_k,N)},\sqrt{\varepsilon_k} W+ v^{(\varepsilon_k,\delta_k,N)})\Big.\cr
&&\qquad\qquad\qquad+\Big.\frac{1}{2}\|( u^{(\varepsilon_k,\delta_k,N)}, v^{(\varepsilon_k,\delta_k,N)})\|_{\mathcal{H}}^2\Big]-2\kappa\cr
&\geq& \mathbb{E}\big[F \circ \mathcal{G}^0(u^\infty,v^\infty)+\frac{1}{2}\|({u^\infty},{v^\infty})\|_{\mathcal{H}}^2\big]-2\kappa\cr
&\geq &\quad \inf \big\{F(\phi)+\frac{1}{2}\|({u},{v})\|_{\mathcal{H}}^2:(u,v)\in\mathcal{H},{ \phi=\mathcal{G}^0(u,v)}\big\}-2\kappa\cr
&\geq& \inf _{\phi \in \mathcal{E}}\{F(\phi)+I(\phi)\}-2\kappa
\end{eqnarray*}
for every $\kappa \in (0,1).$
In the second inequality, we used the following facts; 
{\rm (i)} For real sequence $\{a_k\}$, $\{b_k\}$, we have
$\liminf_{k\to \infty}(a_k+b_k)
=\lim\limits_{k\to \infty}a_k
+\liminf_{k\to \infty}b_k$ if $\lim\limits_{k\to \infty} a_k$ exists. {\rm (ii)} In Step 2, we showed $\mathcal{G}^{(\varepsilon_k,\delta_k)}(\sqrt{\varepsilon_k} B^H+ u^{(\varepsilon_k,\delta_k,N)},\sqrt{\varepsilon_k} W+ v^{(\varepsilon_k,\delta_k,N)})$ weakly converges to $\mathcal{G}^0(u^\infty,v^\infty)$ as $k\to \infty$. {\rm (iii)} Fatou's lemma and lower semi-continuity of $\frac{1}{2}\|(\cdot,\cdot)\|^2_{\mathcal{H}}$. (In {\rm (iii)} we also used Skorohod's theorem \cite[Page 9, Theorem 2.7]{1989Ikeda} to turn a weakly convergent sequence of $\mathcal{H}$-valued random variables into an almost convergent one on another probability space without changing their laws.)

Finally, we will show that  $I:\mathcal{E} \to [0,\infty]$ is a good rate function, that is, the sublevel set $\{\phi: I(\phi) \leq M\}$ is compact for any $0\le M<\infty$. 

We first define the set $\Gamma_N=\left\{\mathcal{G}^0(u,v)\in \mathcal{E}: (u,v) \in S_N\right\}$.  According to  Step 1, we can see the 
set $\Gamma_N=\left\{\mathcal{G}^0(u,v)\in \mathcal{E}: (u,v) \in S_N\right\}$ is compact.
It is sufficient to prove that 
\begin{eqnarray}\label{5-20}
\{\phi: I(\phi) \leq M\}=\bigcap_{n=1}^{\infty} \Gamma_{M+\frac{1}{n}} . 
\end{eqnarray}
Firstly, we will show that $\{\phi: I(\phi) \leq M\}\subseteq\bigcap_{n=1}^{\infty} \Gamma_{M+\frac{1}{n}} $. Let $\phi\in \mathcal{E}$ such that $I(\phi)\le M$. By the definition of $I(\phi) $, for each $n$, there exists $(u^n,v^n)\in \mathcal{H}$ such that $\frac{1}{2}\|u^n\|_{\mathcal{H}^{H,d_1}}^2\le M+\frac{1}{n}$. So we have  $\phi\in \Gamma_{M+\frac{1}{n}}$. Next, we will show that $\{\phi: I(\phi) \leq M\}\supseteq\bigcap_{n=1}^{\infty} \Gamma_{M+\frac{1}{n}} $. Assume $\phi\in\Gamma_{M+\frac{1}{n}}$ for all $n$. For each $n$, there exists $(u^n,v^n)\in S_{M+\frac{1}{n}}$ such that $\phi=\mathcal{G}^0(u^n,v^n)$. So for all $n$, we have $I(\phi) \leq \frac{1}{2}\left\|u^n\right\|_{\mathcal{H}^{H,d_1}}^2 \leq M+\frac{1}{n}$. Let $n \to \infty$, we get $I(\phi) \leq  M$. 
Thus, we proved (\ref{5-20}).

According to the equivalence between the Laplace principle and large deviation (\cite[Theorem 1.8]{BP_book} for example),  
we have that the slow variable $x^{(\varepsilon,\delta)}$ of system (\ref{1}) satisfies a large deviation principle with the rate function $I: C^{\alpha}\left([0,T], \mathbb{R}^m\right)\rightarrow [0, \infty]$. 
This proof of our main theorem (\thmref{thm}) was completed.\qed

	\section*{Acknowledgments}
This work was partly supported by the Key International (Regional) Cooperative Research Projects of the NSF of China under Grant No. 12120101002, the NSF of China under Grant No. 12072264. 
{and JSPS KAKENHI (Grant No. 20H01807).}

	\section*{Declarations}
	The authors declare that they have no known competing financial interests or personal relationships that could have appeared
	to influence the work reported in this paper.


\end{document}